\documentclass[12pt,reqno]{amsart}

\usepackage{amssymb,amsmath,graphicx,amsfonts,euscript}
\usepackage{color}

\newtheorem{thm}{Theorem}[section]

\newtheorem{prop}[thm]{Proposition}
\newtheorem{define}[thm]{Definition}

\newtheorem{lemma}[thm]{Lemma}

\def\om{\omega}

\def\pp{\partial}
\def\na{\nabla}
\def\th{\theta}

\numberwithin{equation}{section}
\subjclass[2000]{35Q35, 35B35, 35B65, 76D03}
\keywords{Critical Boussinesq equations, global regularity}

\begin{document}
\title[The 2D Boussinesq equations with critical dissipation]{The 2D incompressible Boussinesq equations with general critical dissipation}

\author[Q. Jiu, C. Miao, J. Wu and Z. Zhang]{Quansen Jiu$^{1}$, Changxing Miao$^{2}$, Jiahong Wu$^{3}$ and Zhifei Zhang$^{4}$}

\address{$^1$ School of Mathematical Sciences,
Capital Normal University,
Beijing 100048, People's Republic of China}

\email{jiuqs@mail.cnu.edu.cn}

\address{$^2$ Institute of Applied Physics and Computational Mathematics,
P.O. Box 8009, Beijing 100088, People's Republic of China}

\email{miao\_changxing@iapcm.ac.cn}

\address{$^3$ Department of Mathematics,
Oklahoma State University,
401 Mathematical Sciences,
Stillwater, OK 74078, USA}

\email{jiahong@math.okstate.edu}

\address{$^4$ School of Mathematical Sciences,
Peking University, Beijing 100871, People's Republic of China}

\email{zfzhang@math.pku.edu.cn}

\date{\today}

\begin{abstract}
This paper aims at the global regularity problem concerning the 2D incompressible Boussinesq equations with general critical dissipation. The critical dissipation
refers to $\alpha +\beta=1$ when $\Lambda^\alpha \equiv (-\Delta)^{\frac{\alpha}{2}}$ and $\Lambda^\beta$ represent
the fractional Laplacian dissipation in the velocity and the temperature equations, respectively. We establish the global regularity for the general
case with $\alpha+\beta=1$ and $0.9132\approx \alpha_0<\alpha<1$. The cases when $\alpha=1$ and when $\alpha=0$ were previously resolved by Hmidi, Keraani and Rousset \cite{HKR1,HKR2}. The global existence and uniqueness is achieved here by
exploiting the global regularity of a generalized critical surface quasi-gesotrophic
equation as well as the regularity of a combined quantity of the vorticity and
the temperature.
\end{abstract}
\maketitle

\vskip .1in
\section{Introduction}
\label{intro}

This paper studies the global (in time) regularity of solutions to the 2D incompressible Boussinesq equations with a general critical dissipation
\begin{eqnarray}\label{BQE}
\begin{cases}
\partial_t u + u\cdot \nabla u + \nu\, \Lambda^\alpha u
=- \nabla p + \theta \mathbf{e}_2, \qquad x\in \mathbb{R}^2, \,\, t>0, \\
\nabla \cdot u=0, \qquad x\in \mathbb{R}^2, \,\, t>0, \\
\partial_t \theta + u\cdot \nabla \theta
+ \kappa\, \Lambda^{\beta}\theta =0,\qquad x\in \mathbb{R}^2, \,\, t>0,\\
u(x,0) =u_0(x),\,\, \th(x,0) =\th_0(x), \qquad x\in \mathbb{R}^2,
\end{cases}
\end{eqnarray}
where $u=u(x,t)$ denotes the 2D velocity, $p=p(x,t)$ the pressure,
$\th=\th(x,t)$ the temperature, $\mathbf{e}_2$ the unit vector in the vertical
direction, and $\nu>0$, $\kappa>0$, $0<\alpha<1$ and $0<\beta<1$
are real parameters. Here $\Lambda= \sqrt{-\Delta}$ represents the Zygmund operator with $\Lambda^\alpha$ being defined through the Fourier transform, namely
$$
\widehat{\Lambda^\alpha f}(\xi) = |\xi|^\alpha \,\widehat{f}(\xi),
$$
where the Fourier transform is given by
$$
\widehat{f}(\xi) = \int_{\mathbb{R}^2} e^{-i x\cdot \xi} \, f(x) \,dx.
$$
(\ref{BQE}) generalizes the standard 2D Boussinesq equations in which the dissipation
is given by the Laplacian operator. The 2D Boussinesq equations and their fractional
Laplacian generalizations have attracted considerable attention recently
due to their physical applications and mathematical significance. The Boussinesq
equations model geophysical flows such as atmospheric fronts and oceanic
circulation, and play an important role in the study of
Raleigh-Bernard convection (see, e.g., \cite{Con_D,Gill,Maj,Pe}). Mathematically
the 2D Boussinesq equations serve as a lower dimensional model of
the 3D hydrodynamics equations. In fact, the Boussinesq equations retain
some key features of the 3D Navier-Stokes and the Euler equations such as the
vortex stretching mechanism. As pointed out in \cite{MB}, the inviscid Boussinesq
equations can be identified with the 3D Euler equations for axisymmetric flows.

\vskip .1in
One main focus of recent research on the 2D Boussinesq equations has been on the global
regularity issue when the dissipation is given by a fractional Laplacian or is present
only in one direction (see, e.g., \cite{Abi,ACW10,ACW11,CaDi,CaoWu1,Ch,ChaeWu,CV,Jiu,DP2,DP3,ES,Hmidi,
HmKe1,HmKe2,HKR1,HKR2,HL,KRTW,KRTW1,LaiPan,LLT,MX,Mof,Oh,Wu_Xu,Xu}). Most recent work targets
the critical and the supercritical cases. The critical case refers to $\alpha+\beta=1$ in (\ref{BQE}) while the supercritical case corresponds to $\alpha+\beta<1$.
In \cite{HKR1,HKR2} Hmidi,
Keraani and Rousset were able to establish the global regularity for
two critical cases: \eqref{BQE} with $\alpha=1$ and $\kappa=0$ and \eqref{BQE} with $\nu=0$ and $\beta=1$. Miao and Xue in \cite{MX} obtained
the global regularity for \eqref{BQE} with $\nu>0$, $\kappa>0$ and
$$
\alpha \in \left(\frac{6-\sqrt{6}}{4},\,1\right), \quad  \beta \in \left(1-\alpha, \, \min\left\{\frac{7+2\sqrt{6}}{5}\alpha -2, \frac{\alpha(1-\alpha)}{\sqrt{6}-2\alpha}, 2-2\alpha\right\}\right).
$$
In addition, Constantin and Vicol \cite{CV} verified the global regularity of \eqref{BQE} with
$$
\nu>0, \quad \kappa>0, \quad \alpha\in (0,2), \quad \beta\in (0,2), \quad \beta>\frac{2}{2+\alpha}.
$$
We briefly mention that the logarithmically supercritical case has also been
dealt with and global
regularity has been established (\cite{ChaeWu,Hmidi,KRTW,KRTW1}). This paper aims at the global regularity
of \eqref{BQE} with a general critical dissipation
$$
0<\alpha<1, \quad 0<\beta<1, \quad \alpha+ \beta =1
$$
and we succeed in the case when $\alpha_0<\alpha<1$, where $\alpha_0$ is
given in \eqref{alpha0} below. More precisely we are able to prove the
following theorem.

\begin{thm} \label{main}
Let $\alpha_0<\alpha<1$ and $\alpha+\beta=1$, where
\begin{equation}\label{alpha0}
\alpha_0 = \frac{23-\sqrt{145}}{12} \approx 0.9132.
\end{equation}
Assume that  $u_0 \in B^\sigma_{2,1}(\mathbb{R}^2)$ with $\sigma\ge \frac52$ and $\th_0 \in B^2_{2,1}(\mathbb{R}^2)$. Then \eqref{BQE} has
a unique global solution $(u, \th)$ satisfying,
for any $0<T<\infty$,
\begin{equation} \label{solutionclass}
\begin{split}
&u \in C([0,T]; B^\sigma_{2,1}(\mathbb{R}^2))\cap
L^1([0,T]; B^{\sigma+\alpha}_{2,1}(\mathbb{R}^2)), \\
&\th\in C([0,T]; B^2_{2,1}(\mathbb{R}^2))\cap L^1([0,T]; B^{2+\beta}_{2,1}(\mathbb{R}^2)).
\end{split}
\end{equation}
\end{thm}

Here $B^s_{q,r}$ with $s\in \mathbb{R}$ and $q,r\in [1,\infty]$ denotes an inhomogeneous Besov space
and its precise definition
is given in Section \ref{BesovComm}. The key component in the proof of Theorem \ref{main} is to establish
the global {\it a priori} bounds in the class defined in \eqref{solutionclass}.
This does not appear to be trivial and the energy methods are not
sufficient for this purpose. Although the global bounds for $u$ in $L^\infty([0,T]; L^2)$ and $\th$ in $L^\infty([0,T]; L^q)$ with $q\in
[2,\infty]$ can be easily obtained, the global bounds for the derivatives are not evident. To
avoid the pressure term, we resort to the vorticity formulation
\begin{eqnarray}\label{Veq}
\begin{cases}
\partial_t \om +u\cdot \nabla \om + \nu \Lambda^\alpha \om= \partial_{1}\theta , \\
u=\nabla^\perp \psi, \quad \Delta \psi = \omega \qquad \mbox{or}\quad
u =\nabla^\perp \Delta^{-1} \om.
\end{cases}
\end{eqnarray}
However, the ``vortex stretching" term $\pp_1 \th$ appears to prevent us from
proving any global bound for $\om$. A natural idea would be to
eliminate $\pp_1 \th$ from the vorticity equation. For notational convenience, we set $\nu=\kappa=1$ in (\ref{BQE}) throughout the rest of this paper. Realizing that
$\Lambda^\alpha \om - \pp_1 \th= \Lambda^\alpha (\om-\Lambda^{-\alpha} \pp_1\th)$, we can hide $\pp_1\th$ by
considering the new quantity
$$
G = \om - \mathcal{R}_\alpha \th \quad \mbox{with}\quad \mathcal{R}_\alpha = \Lambda^{-\alpha} \pp_1,
$$
which satisfies
\begin{equation}\label{Geqin}
\partial_t G + u\cdot \nabla G + \Lambda^\alpha G= [\mathcal{R}_\alpha, u\cdot\na]\theta +\Lambda^{\beta-\alpha}\partial_{1}\theta.
\end{equation}
Here we have used the standard commutator notation
$$
[\mathcal{R}_\alpha, u\cdot\na]\theta
= \mathcal{R}_\alpha(u\cdot\na\theta) - u\cdot\na \mathcal{R}_\alpha\th.
$$
(\ref{Geqin}) can be obtained by taking the difference of the vorticity equation and
the resulting equation after applying $\mathcal{R}_\alpha$ to the temperature equation.
The quantity $G$ was first introduced in \cite{HKR1} to deal with the critical case when
$\alpha=1$ and $\beta=0$. Although (\ref{Geqin}) appears to be more complicated than
the vorticity equation, but the commutator term $[\mathcal{R}_\alpha, u\cdot\na]\theta$
is less singular than $\pp_1\th$ in the vorticity equation. By obtaining a suitable
bound for $[\mathcal{R}_\alpha, u\cdot\na]\theta$, we are able to obtain a global bound
for $\|G\|_{L^2}$ when $\alpha>\frac45$. In addition, by fully exploiting the
dissipation with $\alpha>\frac45$, a global bound is also established for $\|G\|_{L^q}$ when $q$ is in the range
\begin{equation}\label{qrange}
2 < q < q_0 \equiv \frac{8-4\alpha}{8-7\alpha}.
\end{equation}
This global bound for $\|G\|_{L^q}$ enables us to gain further regularity for $G$. In
fact, we establish that, for $\alpha_0<\alpha$ and $s\le 3\alpha-2$, the Besov norm $\|G\|_{B^{s}_{q,\infty}}$ obeys, for any $T>0$ and $t\le T$,
\begin{equation}\label{GBe}
\|G(t)\|_{B^{s}_{q,\infty}} \le C,
\end{equation}
where $C$ is a constant depending on $T$ and the norms of the initial data.

\vskip .1in
In contrast to the critical case with $\alpha=1$ and $\kappa=0$ dealt with in \cite{HKR1},
the general critical case appears to be more difficult.  The regularity of
$G$ here does not translate to the regularity on the vorticity $\om$ since the
corresponding regularity of
$\mathcal{R}_\alpha \th$ is not known for $\alpha<1$. This paper offers
a different approach by gaining further regularity through the temperature equation
\begin{equation}\label{theq}
\pp_t \th + u\cdot\na\th + \Lambda^{\beta} \th=0.
\end{equation}
Since $u$ is determined by $\om$ through the Biot-Savart law,
or $u =\nabla^\perp\Delta^{-1} \om$ and $\om =G + \mathcal{R}_\alpha \th$ ,
we can decompose $u$ into two parts,
\begin{equation}\label{utheq}
u= \nabla^\perp \Delta^{-1} \om =
\nabla^\perp \Delta^{-1} G
+ \nabla^\perp \Delta^{-1} \mathcal{R}_\alpha \th \equiv \widetilde{u} + v.
\end{equation}
For $\alpha_0<\alpha$ and as a consequence of (\ref{GBe}), $\widetilde{u}$ is regular in the sense that
$$
\|\nabla \widetilde{u}\|_{L^\infty}
= \|\nabla \nabla^\perp \Delta^{-1} G\|_{L^\infty} \le C\, \|G\|_{B^{s}_{q,\infty}} \le C.
$$
In addition, when $\alpha+\beta=1$, $v$ in terms of $\th$ can be written as
$$
v =\nabla^\perp \Delta^{-1} \Lambda^{-(1-\beta)}\, \pp_1\,\th.
$$
Therefore, \eqref{theq} is almost a generalized critical surface
quasi-geostrophic (SQG) type equation first studied in \cite{CIW}
except that $u$ here contains
an extra regular velocity $\widetilde{u}$. We remark that there is a large literature
on the SQG equation and interested readers may consult \cite{CaV,CMZ,Constan,CIW,CMTa,CV,CC,CorF,KNV} and the references therein. Since energy estimates do
not appear to yield the desired global {\it a priori} bounds, we employ the
approach of Constantin and Vicol \cite{CV} to establish the global regularity of
(\ref{theq}) and (\ref{utheq}). Different from the Schwartz class setting in \cite{CV},
the initial data $\th_0$ here is in $H^1$ with $\|\na \th_0\|_{L^\infty} <\infty$. The precise global existence and uniqueness
of (\ref{theq}) and (\ref{utheq}) obtained here can be stated as follows.

\begin{thm} \label{glohs}
Let $\beta\in (0,1)$ and $0<T<\infty$. Let $\widetilde{u}$ be a 2D vector field satisfying $\nabla\cdot \widetilde{u}=0$ and
$$
M\equiv \max\left\{\|\widetilde{u}\|_{L^\infty(0,T; L^2(\mathbb{R}^2))}, \|\nabla \widetilde{u}\|_{L^\infty(0,T; L^\infty(\mathbb{R}^2))}\right\} < \infty.
$$
Consider the generalized critical SQG type equation
\begin{equation} \label{active}
\begin{cases}
\pp_t \theta + u\cdot\na \theta + \Lambda^{\beta} \theta =0, \quad x\in \mathbb{R}^2, \,\, t>0, \\
u = \widetilde{u} + v, \quad v =-\nabla^\perp
\Lambda^{-3+\beta} \pp_{1} \theta, \quad x\in \mathbb{R}^2, \,\,t>0, \\
\theta(x,0) =\th_0(x), \quad x\in \mathbb{R}^2.
\end{cases}
\end{equation}
Assume that $\th_0\in H^1(\mathbb{R}^2)$ with
$$
\|\nabla\theta_0\|_{L^\infty} <\infty.
$$
Then (\ref{active}) has a unique
solution $\th\in C([0,T]; H^1(\mathbb{R}^2))$ satisfying
$$
\|\nabla \th\|_{L^\infty(0,T; L^\infty(\mathbb{R}^2))}  \le C(M, \|\th_0\|_{H^1}, \|\na \th_0\|_{L^\infty},T).
$$
\end{thm}

In order to prove this theorem, we need to convert the operator
relating $v$ and $\pp_{1}\th$,
namely $\nabla^\perp \Lambda^{-3+\beta}$ into an integral form.
Since $\beta\in (0,1)$, the standard Riesz potential formula does not appear to
apply here (see, e.g., \cite{St}). Nevertheless,  $\nabla^\perp \Lambda^{-3+\beta}$
can be represented through an integral kernel by making use of the inverse Fourier
transform of functions of the form $\frac{P_k(\xi)}{|\xi|^{k+2-\beta}}$, where $P_k$ is
a harmonic polynomial of degree $k$ (\cite[p.73]{St}). As a special consequence, the symmetric
part of $\na v$ can be represented as
$$
S(\na v) \equiv \frac12(\na v + (\na v)^t) = C \int_{\mathbb{R}^2} \frac{\sigma(x-y)}{|x-y|^{1+\beta}}\,(\pp_1 \th(x) -\pp_1 \th(y))\,dy
$$
with
$$
\sigma(z) =\frac{1}{|z|^2} \left[\begin{array}{cc} -2z_1 z_2 & z_1^2-z_2^2\\ z_1^2-z_2^2& 2z_1 z_2 \end{array}\right].
$$
More details can be found in Section \ref{repre}.

\vskip .1in
The gained regularity in $\th$ via Theorem \ref{glohs} allows us to assert
the desired regularity in the velocity and the vorticity. Especially,
$\|\om\|_{L^\infty}$ is bounded
on any time interval $[0, T]$. Further regularity leading to \eqref{solutionclass}
is established through energy estimates in Besov space settings. With these global
bounds at our disposal, the global existence part of Theorem \ref{main} follows from
a local existence through a standard procedure such as the successive approximation and
an extension of the local solution into a global one with the aid of the global {\it a priori} bounds. The uniqueness of solutions in the class \eqref{solutionclass} is clear.

\vskip .1in
The rest of this paper is organized as follows. The second section represents the
relation $v =\nabla^\perp \Delta^{-1} \Lambda^{-(1-\beta)}\, \pp_1\,\th$ as an integral.
Integral formulas for $\na v$ and its symmetric part are also derived in this section. The third
section proves Theorem \ref{glohs}. The fourth section provides the definitions of
functional spaces such as the Besov spaces and related facts. In addition, a commutator
estimate in the Besov space setting is also proven in this section. This commutator
estimate will be used extensively in the sections that follow. Sections \ref{GL2} and
\ref{GLq} establish global {\it a priori} bounds for $\|G\|_{L^2}$ and
for $\|G\|_{L^q}$, where $q$ satisfies (\ref{qrange}). The last section proves Theorem
\ref{main}. To do so, we first obtain a global bound of $G$ in the Besov space $B^s_{q,\infty}$ with any $s\le 3\alpha-2$. This global bound and
Theorem \ref{glohs} yield the proof of Theorem \ref{main}.

\vskip .4in
\section{Representing $v=\na^\perp\Lambda^{-3-\beta}\pp_1\th$
and $\nabla v$ as integrals}
\label{repre}

In this section we represent $v=\na^\perp\Lambda^{-3-\beta}\pp_1\th$, its gradient $\na v$ and the symmetric part of $\na v$ as integrals. These integral representations will be used in the next section. More precisely, we prove the following lemma.

\begin{lemma} \label{vrep}
Let $\theta$ be a smooth function of $\mathbb{R}^2$ which is sufficiently rapidly
decreasing at $\infty$. Let $\beta\in (0,1)$ and $v$ be given by
\begin{equation} \label{vdef}
v = \na^\perp\Lambda^{-3-\beta}\pp_1\th.
\end{equation}
Then $v$ and $\na v$ can be written as
\begin{eqnarray}
v(x) &=& C(\beta)\, \int_{\mathbb{R}^2} \frac{(x-y)^\perp}{|x-y|^{1+\beta}}\, \pp_1 \th(y)\,dy, \nonumber\\
\na v(x) &=& C(\beta)\,\left[\begin{array}{cc} 0 & -1\\ 1&0 \end{array} \right]\int_{\mathbb{R}^2} \frac{1}{|x-y|^{1+\beta}}\, \pp_1 \th(y)\,dy \nonumber\\
&& - (1+\beta) C(\beta)\,  \int_{\mathbb{R}^2} \frac{(x-y)^\perp \otimes (x-y)}{|x-y|^{3+\beta}}\,\pp_1 \th(y)\,dy, \label{nav}
\end{eqnarray}
where $C(\beta)$ is a constant depending on $\beta$ only, and $a\otimes b$
denotes the tensor product of two vectors $a$ and $b$, namely $a\otimes b =(a_ib_j)$.
Especially the symmetric part of $\na v$, denoted by $S(\na v)$, is given by
\begin{equation}\label{vsys}
S(\na v) \equiv \frac12(\na v + (\na v)^t) = C \int_{\mathbb{R}^2} \frac{\sigma(x-y)}{|x-y|^{1+\beta}}\,(\pp_1 \th(x) -\pp_1 \th(y))\,dy,
\end{equation}
where $C$ is constant depending on $\beta$ only and
$$
\sigma(z) =\frac{1}{|z|^2} \left[\begin{array}{cc} -2z_1 z_2 & z_1^2-z_2^2\\ z_1^2-z_2^2& 2z_1 z_2 \end{array}\right].
$$
\end{lemma}

\vskip .1in
To derive the formulas in Lemma \ref{vrep}, we need a fact stated in the
following lemma, which can be found in Stein's book (\cite[p.73]{St}). We note that
the Fourier transform $\widehat{f}(\xi)$ in Stein's book (\cite[p.46]{St}) is defined
by
$$
\widehat{f}(\xi) = \int_{\mathbb{R}^d} e^{2\pi i\, x\cdot \xi} \,f(x) \, dx,
$$
which is normally the definition of the inverse Fourier transform. This explains why
the Fourier transform $\widehat{f}$ in Stein's book is here changed to the inverse Fourier transform $f^\vee$.

\begin{lemma}\label{stl}
Let $P_k(x)$ with $x\in \mathbb{R}^d$ be a homogeneous harmonic polynomial
of degree $k$, where $k\ge 1$ is an integer. Let $0<\alpha <d$. Then, for a constant
$C(d,k,\alpha)$ depending on $d,k$ and $\alpha$ only
$$
\left(\frac{P_k(\xi)}{|\xi|^{d+k-\alpha}}\right)^\vee
= C(d,k,\alpha) \, \frac{P_k(x)}{|x|^{k+\alpha}}
$$
in the sense that
$$
\int_{\mathbb{R}^d}  \frac{P_k(\xi)}{|\xi|^{d+k-\alpha}}\, \phi(\xi)\, d\xi
= C(d,k,\alpha) \,\int_{\mathbb{R}^d} \frac{P_k(x)}{|x|^{k+\alpha}}\,
\phi^\vee(x)\,dx
$$
for every $\phi$ which is sufficiently rapidly
decreasing at $\infty$, and whose inverse Fourier transform has the same property.
\end{lemma}

\vskip .1in
With this lemma at our disposal, we can now prove Lemma \ref{vrep}.
\begin{proof}[Proof of Lemma \ref{vrep}]
By (\ref{vdef}),
$$
\widehat{v}(\xi) = i\,\xi^\perp |\xi|^{-3+\beta} \, \widehat{\pp_1\th} (\xi).
$$
According to Lemma \ref{stl},
$$
v(x) = i\,\left(\xi^\perp |\xi|^{-3+\beta}\right)^\vee \ast \pp_1\th = C(\beta) \frac{x^\perp}{|x|^{1+\beta}}\ast \pp_1\th =C(\beta)\,\int_{\mathbb{R}^2} \frac{(x-y)^\perp}{|x-y|^{1+\beta}}\, \pp_1 \th(y)\,dy.
$$
Since $\beta\in (0,1)$, the kernel in the representation of $v$ is not singular and
we have
\begin{eqnarray*}
\na v &=& C(\beta) \,\int_{\mathbb{R}^2} \na_x\left[\frac{(x-y)^\perp}{|x-y|^{1+\beta}}\right]\, \pp_1 \th(y)\,dy\\
&=& C(\beta)\,\left(\begin{array}{cc} 0 & -1\\ 1&0 \end{array} \right)\int_{\mathbb{R}^2} \frac{1}{|x-y|^{1+\beta}}\, \pp_1 \th(y)\,dy \nonumber\\
&& - (1+\beta) C(\beta)\,  \int_{\mathbb{R}^2} \frac{(x-y)^\perp \otimes (x-y)}{|x-y|^{3+\beta}}\,\pp_1 \th(y)\,dy.
\end{eqnarray*}
(\ref{vsys}) is obtained by taking the symmetric part of (\ref{nav}) and inserting $\pp_1\th(x)$ in the resulting integral. The inserted term does not contribute
to the integral thanks to the fact that, for any $r>0$,
$$
\int_{|z|=r} \sigma(z)\,dz =0.
$$
This completes the proof of Lemma \ref{vrep}.
\end{proof}

\vskip .4in
\section{Global regularity for an active scalar with critical dissipation}

The goal of this section is to prove Theorem \ref{glohs}, which states the global
regularity of a generalized SQG type equation with critical
dissipation. The proof is obtained by modifying the approach of Constantin and Vicol \cite{CV}. Different from \cite{CV}, the functional setting here is weaker.

\vskip .1in
To prove Theorem \ref{glohs}, we first establish
the global existence and uniqueness
of (\ref{active}) when $\widetilde{u}$ is smooth and $\th_0$ is smooth
and decays sufficiently fast at infinity.

\begin{thm} \label{glreg}
Let $\beta\in (0,1]$ and $0<T<\infty$. Let $\theta_0 \in C^\infty(\mathbb{R}^2)$
and decays sufficiently fast at $\infty$. Let $\widetilde{u}$ be a smooth 2D vector field satisfying $\nabla\cdot \widetilde{u}=0$ and
$$
M\equiv \max\left\{\|\widetilde{u}\|_{L^\infty(0,T; L^2(\mathbb{R}^2))}, \|\nabla \widetilde{u}\|_{L^\infty(0,T; L^\infty(\mathbb{R}^2))}\right\} <\infty.
$$
Then (\ref{active}) has a unique global smooth solution $\theta$ on $[0,T]$. In addition,
\begin{equation} \label{nathb}
\|\nabla \th\|_{L^\infty(0,T; L^\infty(\mathbb{R}^2))}  \le C(M, \|\th_0\|_{L^\infty}, \|\na \th_0\|_{L^\infty}, T),
\end{equation}
where the constant $C$ in the inequality above depends only on the quantities inside the parenthesis.
\end{thm}

We will be more specific about the decay requirement
on $\theta_0$ in the proof of Theorem \ref{glreg}.
The rest of this section proves Theorem \ref{glreg} and then Theorem \ref{glohs}.
The proof of Theorem \ref{glreg} employs the method of Constantin and Vicol \cite{CV}.
We recall a basic concept. For a given $\delta>0$, a function $f$ is said to have only small shocks with a parameter $\delta$ or simply $f\in OSS_\delta$ if there is $L>0$ such that
\begin{equation}\label{delL}
|f(x) -f(y)| \le \delta \quad\mbox{whenever $|x-y| <L$}.
\end{equation}
The proof of Theorem \ref{glreg} consists of two main parts. The first part shows that $\th\in OSS_\delta$
for a suitable $\delta=\delta(\|\th_0\|_{L^\infty})$ implies (\ref{nathb}).
$L$ in (\ref{delL}) is not required to be big in order for this part of
the result to hold.  What we really need here is a smoothness property on $\theta$
and it suffices for $\theta\in OSS_\delta$ with a small $L$, say $L<1$.
The second part proves that $\theta_0\in OSS_{\frac{\delta}{4}}$ for some $L>0$ implies that $\th\in OSS_\delta$ for the same $L$.
For the sake of clarity , we present each part as a proposition.

\vskip .1in
\begin{prop} \label{xxx1}
Let $\beta\in (0,1]$ and $0<T<\infty$. Let $\widetilde{u}$ be a vector field satisfying $\nabla\cdot \widetilde{u}=0$ and
$$
M\equiv \max\left\{\|\widetilde{u}\|_{L^\infty(0,T; L^2(\mathbb{R}^2))}, \|\nabla \widetilde{u}\|_{L^\infty(0,T; L^\infty(\mathbb{R}^2))}\right\} <\infty.
$$
If $\th$ is in $OSS_{\delta}$ uniformly on $[0,T]$ with $\delta$ given by
\begin{equation} \label{deldef}
\delta = \frac{C}{\|\th_0\|_{L^\infty}^{\frac{2\beta}{2-\beta}}}
\end{equation}
for a suitable pure constant $C>0$ (independent of $\theta_0$), then $\na\th$ is bounded as in (\ref{nathb}).
\end{prop}

\vskip .1in
\begin{prop} \label{xxx2}
Let $\beta\in (0,1]$ and $0<T<\infty$. Let $\widetilde{u}$ be a vector field satisfying $\nabla\cdot \widetilde{u}=0$ and
$$
M\equiv \max\left\{\|\widetilde{u}\|_{L^\infty(0,T; L^2(\mathbb{R}^2))}, \|\nabla \widetilde{u}\|_{L^\infty(0,T; L^\infty(\mathbb{R}^2))}\right\} <\infty.
$$
Assume that $\th_0\in OSS_{\delta/4}$ for some $\delta>0$. Then $\th\in OSS_{\delta}$ uniformly on $[0,T]$.
\end{prop}

\vskip .1in
With the two propositions above in our disposal, we can now prove Theorem \ref{glreg}.
\begin{proof}[Proof of Theorem \ref{glreg}]
Since $\th_0\in C^\infty(\mathbb{R}^2)$ and decays sufficiently fast at $\infty$, it is
easy to check $\th_0\in OSS_{\delta/4}$ for $\delta$ given by (\ref{deldef}). There are
two different ways to achieve this. The first is to use the simple inequality
$$
|\th_0(x)-\th_0(y)| \le \|\na \th_0\|_{L^\infty} |x-y|
$$
and take $L= \delta/(4\|\na\th_0\|_{L^\infty})$. Alternatively,
we first take $R>0$ such that $|\th_0(x)| < \delta/8$ for any $|x|\ge R$. Then, by the
uniform continuity of $\th_0$ in the disk $|x|\le 2R$, there is $L_1>0$ such that
$$
|\th_0(x)-\th_0(y)| \le \delta/4
$$
when $|x-y|\le L_1$. Taking $L=\min\{R, L_1\}$, we obtain $\th_0\in OSS_{\delta/4}$.
By Proposition \ref{xxx2}, $\th\in OSS_\delta$ uniformly on $[0,T]$. By Proposition \ref{xxx1}, $\na\th$ satisfies (\ref{nathb}). The
global existence and uniqueness follows
from a local well-posedness and an extension to a global solution through the
global bound for $\|\na\th\|_{L^\infty}$ in (\ref{nathb}). We omit further details. This completes the proof of Theorem \ref{glreg}.
\end{proof}

\vskip .1in
\begin{proof}[Proof of Theorem \ref{glohs}]
We first regularize $\widetilde{u}$ and the initial data. For $\epsilon>0$, we define $\rho_\epsilon$ to  be the standard mollifier, namely
$$
\rho(x) =\rho(|x|) \in C_0^\infty(\mathbb{R}^2), \quad \rho\ge 0, \quad \int_{\mathbb{R}^2}\rho(x)\,dx =1, \quad
\rho_\epsilon(x) = \epsilon^{-2} \rho\left(\frac{x}{\epsilon}\right).
$$
In addition, let $\chi_\epsilon$ be the standard smooth cutoff, namely $\chi_\epsilon(x)=\chi(\epsilon x), \chi\in C_0^\infty(\mathbb{R}^2)$ and $\chi(x)=1$ in $|x|\le 2$. Now we define
$$
\widetilde{u}^\epsilon =\rho_\epsilon\ast \widetilde{u}, \qquad
\theta_0^\epsilon= \rho_\epsilon\ast (\chi_\epsilon\th_0).
$$
and consider
the following regularized initial-value problem
\begin{equation} \label{active_modified}
\begin{cases}
\pp_t \theta^\epsilon + u^\epsilon\cdot\na \theta^\epsilon
+ \Lambda^{\beta} \theta^\epsilon =0, \quad x\in \mathbb{R}^2, \,\, t>0, \\
u^\epsilon = \widetilde{u}^\epsilon + v^\epsilon, \quad v^\epsilon =-\nabla^\perp
\Lambda^{-3+\beta} \pp_{1} \theta^\epsilon, \quad x\in \mathbb{R}^2, \,\,t>0, \\
\theta^\epsilon(x,0) =\th^\epsilon_0(x), \quad x\in \mathbb{R}^2.
\end{cases}
\end{equation}
Then $\th^\epsilon_0\in C^\infty(\mathbb{R}^2)$ and has the decay properties required in
Theorem \ref{glreg}. Therefore, by Theorem \ref{glreg}, (\ref{active_modified}) has a unique global smooth solution $\th^\epsilon$ satisfying
$$
\|\na \th^\epsilon\|_{L^\infty(0,T; L^\infty(\mathbb{R}^2))} \le C(M, \|\th^\epsilon_0\|_{L^\infty}, \|\na \th^\epsilon_0\|_{L^\infty},T).
$$
Since
$$
\|\th^\epsilon_0\|_{L^\infty} \le \|\th_0\|_{L^\infty}
\le \sqrt[3]{6}\, \|\na\th_0\|_{L^\infty}^{\frac13}\, \|\th_0\|_{L^2}^{\frac13}\,\|\na \th_0\|_{L^2}^{\frac13}
$$
and $\|\na \th^\epsilon_0\|_{L^\infty} \le \|\na\th_0\|_{L^\infty}$, $\na \th^\epsilon$
admits a global bound that is uniform with respect to
 $\epsilon$,
$$
\|\na \th^\epsilon\|_{L^\infty(0,T; L^\infty(\mathbb{R}^2))} \le C(M, \|\th_0\|_{H^1}, \|\na \th_0\|_{L^\infty},T).
$$
In addition, a simple energy estimate shows that
\begin{equation} \label{gggl}
\|\th^\epsilon\|_{H^1}\le C(M, \|\th_0\|_{H^1}, \|\na \th_0\|_{L^\infty},T).
\end{equation}
In fact, (\ref{gggl}) follows from the energy inequality
$$
\frac12 \frac{d}{dt} \|\nabla\th^\epsilon\|^2_{L^2} + \|\Lambda^{1+\frac{\beta}{2}} \th^\epsilon\|^2_{L^2} \le (\|\nabla \widetilde{u}^\epsilon\|_{L^\infty}
+ \|\nabla v^\epsilon\|_{L^\infty}) \|\nabla \th^\epsilon\|^2_{L^2},
$$
and the global bounds $\|\nabla \widetilde{u}^\epsilon\|_{L^\infty} \le \|\nabla \widetilde{u}\|_{L^\infty}$ and
\begin{eqnarray}
\|\nabla v^\epsilon\|_{L^\infty} &=& \|\na \nabla^\perp
\Lambda^{-3+\beta} \pp_{1} \theta^\epsilon\|_{L^\infty} \nonumber \\
&\le& C (\|\theta^\epsilon\|_{L^2} + \|\nabla \th^\epsilon\|_{L^\infty}) \label{simin}\\
&\le& C(M, \|\th_0\|_{H^1}, \|\na \th_0\|_{L^\infty},T). \nonumber
\end{eqnarray}
The interpolation inequality in (\ref{simin}) can
be proven through the
Littlewood-Paley decomposition and the details are deferred to the end of Section \ref{BesovComm}.
Then $\th^\epsilon$ has a weak limit $\theta\in H^1$, namely
\begin{equation} \label{hsweak}
\th^\epsilon \rightharpoonup \th \quad\mbox{weakly in $H^1$\, as \,$\epsilon \to 0$}.
\end{equation}
In addition, it can be shown that
\begin{equation} \label{l2strong}
\th^\epsilon \to \th \quad\mbox{strongly in $L^2$\, as \,$\epsilon \to 0$}.
\end{equation}
This can be achieved by estimating the difference
$\th^{\epsilon_1} -\th^{\epsilon_2}$ through energy estimates. Let $u^{\epsilon_1}$ and
$u^{\epsilon_2}$ be the corresponding velocities. Then $\overline{\th} =\th^{\epsilon_1} -\th^{\epsilon_2}$ and $\overline{u}=u^{\epsilon_1}-u^{\epsilon_2}$ satisfy
\begin{eqnarray*}
\pp_t \overline{\th} + \overline{u}\cdot\na \th^{\epsilon_1}
+ u^{\epsilon_2} \cdot\na \overline{\th}
+ \Lambda^\beta \overline{\th}=0.
\end{eqnarray*}
Therefore,
$$
\frac{d}{dt} \|\overline{\th}\|_{L^2}^2 + 2 \|\Lambda^{\beta/2} \overline{\th}\|_{L^2}^2
= -2\int_{\mathbb{R}^2} \overline{u}\cdot \na \th^{\epsilon_1} \overline{\th} \,dx.
$$
Noticing that
$$
\overline{u}=u^{\epsilon_1}-u^{\epsilon_2} =\widetilde{u}^{\epsilon_1}-\widetilde{u}^{\epsilon_2} + v^{\epsilon_1}-v^{\epsilon_2},
$$
we obtain, by H\"{o}lder's inequality,
\begin{eqnarray*}
\left|\int_{\mathbb{R}^2}
\overline{u}\cdot \na \th^{\epsilon_1} \overline{\th} \,dx\right|
&\le& \|\widetilde{u}^{\epsilon_1}-\widetilde{u}^{\epsilon_2}\|_{L^\infty} \|\na \th^{\epsilon_1}\|_{L^2}\,
\|\overline{\th}\|_{L^2}\\
&& \, + \|v^{\epsilon_1}-v^{\epsilon_2}\|_{L^{\frac{2}{\beta}}}\, \|\na \th^{\epsilon_1}\|_{L^{\frac{2}{1-\beta}}}\,
\|\overline{\th}\|_{L^2}.
\end{eqnarray*}
Due to $\|\widetilde{u}^{\epsilon_1}-\widetilde{u}^{\epsilon_2}\|_{L^\infty} \le C |\epsilon_1-\epsilon_2| M$ and recalling (\ref{gggl}), we have
$$
\|\widetilde{u}^{\epsilon_1}-\widetilde{u}^{\epsilon_2}\|_{L^\infty} \|\na \th^{\epsilon_1}\|_{L^2}\,
\|\overline{\th}\|_{L^2} \le C\,|\epsilon_1-\epsilon_2|\,\|\overline{\th}\|_{L^2}.
$$
where $C=C(M, \|\th_0\|_{H^1}, \|\na \th_0\|_{L^\infty},T)$. Since
$$
v^{\epsilon_1}-v^{\epsilon_2} =-\nabla^\perp \Delta^{-1} \Lambda^{-(1-\beta)} \pp_1\overline{\th},
$$
and by Hardy-Littlewood-Sobolev inequality,
$$
\|v^{\epsilon_1}-v^{\epsilon_2}\|_{L^{\frac{2}{\beta}}} \le C\, \|\Lambda^{-(1-\beta)} \overline{\th}\|_{L^{\frac{2}{\beta}}}
\le C\, \|\overline{\th}\|_{L^2}.
$$
By a simple interpolation inequality,
$$
\|\na \th^{\epsilon_1}\|_{L^{\frac{2}{1-\beta}}} \le \|\na \th^{\epsilon_1}\|^{1-\beta}_{L^2}
\|\na \th^{\epsilon_1}\|^\beta_{L^\infty} \le C(M, \|\th_0\|_{H^1}, \|\na \th_0\|_{L^\infty},T).
$$
Therefore,
$$
\frac{d}{dt} \|\overline{\th}\|_{L^2}^2 + 2 \|\Lambda^{\beta/2} \overline{\th}\|_{L^2}^2
\le C\,|\epsilon_1-\epsilon_2|\,\|\overline{\th}\|_{L^2} + \,C\, \|\overline{\theta}\|_{L^2}^2.
$$
Consequently
\begin{eqnarray*}
\|\overline{\th}(t)\|_{L^2} &\equiv& \|\th^{\epsilon_1}(t) -\th^{\epsilon_2}(t)\|_{L^2} \\ &\le&
C\,(|\epsilon_1-\epsilon_2|\, + \, \|\th_0^{\epsilon_1} -\th_0^{\epsilon_2}\|_{L^2}) e^{Ct} \\
&\le& C\, (|\epsilon_1-\epsilon_2|\, +\|\th_0^{\epsilon_1}-\th_0\|_{L^2} +\|\th_0-\th_0^{\epsilon_2}\|_{L^2})\, e^{Ct}\\
&\le& C\, \max\{\epsilon_1, \epsilon_2\}\, e^{Ct},
\end{eqnarray*}
where $C=C(M, \|\th_0\|_{H^1}, \|\na \th_0\|_{L^\infty},T)$.
This proves (\ref{l2strong}). By the interpolation inequality, for $0<s'<1$,
$$
\|\th^\epsilon - \th\|_{H^{s'}} \le C\,\|\th^\epsilon - \th\|^{1-s'/s}_{L^2}\, \|\th^\epsilon - \th\|^{s'/s}_{H^s},
$$
(\ref{hsweak}) and (\ref{l2strong}) imply that, for  $0<s'<1$,
$$
\th^\epsilon \to \th \quad\mbox{strongly in $H^{s'}$\, as \,$\epsilon \to 0$}.
$$
This strong convergence allows us to show that $\th$ is the
corresponding global solution of \eqref{active} satisfying $\th\in L^\infty ([0,T];H^1)$
and  $\|\na \th\|_{L^\infty} \le C(M, \|\th_0\|_{H^1}, \|\na \th_0\|_{L^\infty},T)$. The continuity of $\th$ in time, namely $\th\in C ([0,T];H^1)$, follows from a standard
approach (see, e.g., \cite[p.111]{MB} or \cite[p.138]{BCD}) and we omit the details. Finally,
it is not hard to verify that such solutions are unique. This completes the proof
of Theorem \ref{glohs}.
\end{proof}

\vskip .1in
We now turn to the proofs of Propositions \ref{xxx1} and \ref{xxx2}. To prove
Proposition \ref{xxx1},  we will make use of the following lower bound obtained in \cite{CV}. This lower bound improves a pointwise inequality of C\'ordoba
and C\'ordoba \cite{CC}.

\begin{lemma} \label{low}
Let $\beta\in (0,2)$ and $q\in [1, \infty]$. Let $f\in C^1(\mathbb{R}^d)$ decay
sufficiently fast at $\infty$.
Then the pointwise lower bound holds,
$$
\nabla f \cdot \Lambda^\beta (\na f) \ge \frac12 \Lambda^\beta \left(|\na f|^2\right) + \frac12 D(\na f) + \frac{|\nabla f|^{2+ \frac{\beta q}{q+d}}}{C_0\,\|f\|^{\frac{\beta q}{q+d}}_{L^q}},
$$
where $C_0=C_0(d, \beta, q)$ is a constant and $D$ is given by the principal value integral
$$
D(g) = C(d,\beta)\, \mbox{P.V.} \int_{\mathbb{R}^d} \frac{|g(x)-g(y)|^2}{|x-y|^{d+\beta}}\,dy \ge 0.
$$
\end{lemma}

\vskip .1in
\begin{proof}[Proof of Proposition \ref{xxx1}]
We show that, if
$$
\theta \in OSS_\delta \quad\mbox{with} \quad \delta =\frac{C}{\|\th_0\|_{L^\infty}^{\frac{2\beta}{2-\beta}}}
$$
for a suitable $C$, then $\na \th$ satisfies (\ref{nathb}).  Clearly, $\nabla \th$ satisfies
$$
\pp_t (\na\th) + u\cdot\na(\na\th) + \Lambda^\beta (\na\th) = -\nabla u (\na \th).
$$
Dotting with $\na \th$ and applying Lemma \ref{low} yield
\begin{eqnarray}
&& \frac12 (\pp_t + u\cdot\na + \Lambda^\beta) |\na \th|^2 + \frac12 D(\nabla \th) +
\frac{|\nabla \th|^{2+\beta}}{C_0\,\|\th_0\|^\beta_{L^\infty}} \nonumber\\
&& \qquad \qquad \qquad \qquad  \le |\na\th \cdot S(\nabla u) \cdot \na\th| \le \|S(\na u)\|_{L^\infty}\, |\na\th(x)|^2, \label{mo}
\end{eqnarray}
where $S(\nabla u)$ denotes the symmetric part of $\na u$, or $S(\na u) = \frac12((\na u) + (\na u)^t)$. Clearly
$$
\|S(\na u)\|_{L^\infty} \le \|\na \widetilde{u}\|_{L^\infty} +\|S(\na v)\|_{L^\infty} \le M + \|S(\na v)\|_{L^\infty}.
$$
Now we bound $\|S(\na v)\|_{L^\infty}$. According to Lemma \ref{vrep},
\begin{eqnarray*}
|S(\nabla v(x))| \le C\, \left|\int_{\mathbb{R}^2} \frac{\sigma(x-y)}{|x-y|^{1+\beta}}\,(\pp_1 \th(x) -\pp_1 \th(y))\,dy \right|.
\end{eqnarray*}
To further the estimate, we split the integral over $\mathbb{R}^2$ into three parts:
$$
|x-y|\le \rho, \quad \rho<|x-y|\le L, \quad |x-y|>L
$$
for $0<\rho<L$ to be specified later. By H\"{o}lder's inequality, the integral over
$|x-y|\le \rho$ is bounded by
\begin{eqnarray*}
\left|\int_{|x-y|\le \rho} \frac{1}{|x-y|^{1+\beta}}\, |\pp_1 \th(x) -\pp_1 \th(y)|\,dy\right| \le \sqrt{D(\na \th)}\, \rho^{1-\frac{\beta}{2}}.
\end{eqnarray*}
Through integration by parts, the integral over $\rho<|x-y|\le L$ is bounded by
\begin{eqnarray*}
\left|\int_{\rho<|x-y|\le L} \frac{1}{|x-y|^{2+\beta}}\, |\th(x)-\th(y)|\,dy\right|
&\le& \delta\,\left|\int_{\rho<|x-y|\le L} \frac{1}{|x-y|^{2+\beta}}\,\,dy\right|\\
&=& C\,\delta\, \rho^{-\beta}.
\end{eqnarray*}
We remark that a rigorous justification of the estimate above can be carried out through
smooth cutoff functions. Again through integration by parts, the integral over $|x-y|>L$ is bounded by
\begin{eqnarray*}
\left|\int_{|x-y|>L} \frac{1}{|x-y|^{2+\beta}}\, |\th(x)-\th(y)|\,dy\right|
\le C\, \|\theta_0\|_{L^\infty} L^{-\beta}.
\end{eqnarray*}
Therefore,
\begin{eqnarray*}
\|S(\na u)\|_{L^\infty}\, |\na\th(x)|^2
&\le & C\,\sqrt{D(\na \th)}\, \rho^{1-\frac{\beta}{2}}\, |\na\th(x)|^2
+ C\,\delta\, \rho^{-\beta} \, |\na\th(x)|^2\\
&& + (M + C\, \|\theta_0\|_{L^\infty} L^{-\beta}) \, |\na\th(x)|^2\\
&\le & \frac12 D(\na \th) + C\,\rho^{2-\beta}\, |\na\th(x)|^4+C\,\delta\, \rho^{-\beta} \, |\na\th(x)|^2\\
&& + (M + C\, \|\theta_0\|_{L^\infty} L^{-\beta}) \, |\na\th(x)|^2,
\end{eqnarray*}
where $C$'s are pure constants depending on $\beta$ only.
If we choose $\rho$ and $\delta$ as
$$
\rho= \frac{1}{(4C\,C_0 \|\th_0\|^\beta_{L^\infty})^{\frac{1}{2-\beta}} \, |\na \th(x)|}
\quad \mbox{and}\quad \delta = \frac{1}{C^{1+\beta}\, (4C_0 \|\th_0\|^\beta_{L^\infty})^{\frac{2}{2-\beta}}},
$$
then
$$
\|S(\na u)\|_{L^\infty}\, |\na\th(x)|^2  \le \frac12 D(\na \th) + \frac{|\nabla \th|^{2+\beta}}{2C_0\,\|\th_0\|^\beta_{L^\infty}} + (M + C\, \|\theta_0\|_{L^\infty} L^{-\beta}) \, |\na\th(x)|^2.
$$
Inserting this bound in (\ref{mo}) yields
\begin{equation} \label{bott}
\frac12 (\pp_t + u\cdot\na + \Lambda^\beta) |\na \th|^2  +
\frac{|\nabla \th(x)|^{2+\beta}}{2C_0\,\|\th_0\|^\beta_{L^\infty}} \le (M + C\, \|\theta_0\|_{L^\infty} L^{-\beta}) \, |\na\th(x)|^2.
\end{equation}
This differential inequality allows us to conclude that
\begin{equation} \label{bbbd}
|\na\th(x,t)| \le \max\{\|\na \th_0\|_{L^\infty}, \widetilde{C} (M, \|\theta_0\|_{L^\infty})\},
\end{equation}
where $\widetilde{C}$ is given by
\begin{equation} \label{ctil}
\widetilde{C} = \left(2C_0(M + C\,\|\theta_0\|_{L^\infty} L^{-\beta} \right)^{\frac{1}{\beta}} \|\theta_0\|_{L^\infty}.
\end{equation}
To show (\ref{bbbd}), we define $\Gamma: [0,\infty)\to [0,\infty)$ to be a nondecreasing $C^2$ convex function satisfying $\Gamma(\rho)= 0$ for $0\le \rho \le \max\{\|\na \th_0\|_{L^\infty}, \widetilde{C}\}$ and $\Gamma(\rho)> 0$
for $\rho > \max\{\|\na \th_0\|_{L^\infty}, \widetilde{C}\}$. Multiplying (\ref{bott})
by $\Gamma'(|\na \th|^2)$ and applying the lower bound
\begin{equation} \label{lowbd}
\Lambda^\beta \left(\Gamma(f)\right) \le \Gamma'(f)\,
\Lambda^\beta\,f,
\end{equation}
where is valid for any $\beta\in(0,2)$ and any convex function $\Gamma$ (see \cite{Constan,CC}), we obtain
\begin{eqnarray*}
&&\frac12 (\pp_t + u\cdot\na + \Lambda^\beta)\, \Gamma(|\na \th|^2)
\\
&& \qquad \qquad \le (M + C\, \|\theta_0\|_{L^\infty} L^{-\beta}) \, |\na\th(x)|^2\left[1-\frac{|\na \th(x,t)|^\beta}{\widetilde{C}^\beta}\right]\, \Gamma'(|\na \th|^2).
\end{eqnarray*}
It is easy to verify that the right-hand side of the inequality above is always
less than or equal to zero due to the definition of $\Gamma$. Therefore,
$$
(\pp_t + u\cdot\na + \Lambda^\beta)\, \Gamma(|\na \th|^2) \le 0.
$$
Thanks to $\nabla \cdot u=0$ and the positivity of integral
$$
\int_{\mathbb{R}^2} f(x) |f(x)|^{2p-2}\,\,\Lambda^\beta\,f(x)\,dx \ge 0,
$$
we obtain $\|\Gamma(|\na \th|^2)\|_{L^{2p}} \le \|\Gamma(|\na \th_0|^2)\|_{L^{2p}}$ for any $1\le p<\infty$. Letting $p\to \infty$, we have
$$
\|\Gamma(|\na \th(t)|^2)\|_{L^\infty} \le \|\Gamma(|\na \th_0|^2)\|_{L^\infty} \le 0,
$$
which implies $(\ref{bbbd})$. This completes the proof of Proposition \ref{xxx1}.
\end{proof}

\vskip .1in
To prove Proposition \ref{xxx2}, we need the lower bound in the following lemma.
This lemma can be proven by following the lines
of that for Lemma \ref{low} (see \cite{CV}).

\begin{lemma}\label{difflow}
Let $\beta\in(0,2)$. Let $h\in \mathbb{R}^2$ and $\delta_h\theta(x,t)
\equiv \th(x+h,t) -\th(x,t)$. Then the following pointwise lower bound holds,
$$
\delta_h \th\,\Lambda^\beta \delta_h \th \ge \frac12 \Lambda^\beta (\delta_h \th)^2
+ D_h(\delta_h \th) + C\, \frac{|\delta_h \th(x,t)|^{2+\beta}}{\|\theta\|_{L^\infty}^\beta |h|^\beta},
$$
where $C$ is a constant depending on $\beta$ only, and
$$
D_h(\delta_h \th) = c\, P.V. \int_{\mathbb{R}^2} \frac{(\delta_h \th(x,t)-\delta_h \th(y,t))^2}{|x-y|^{2+\beta}} \,dy.
$$
\end{lemma}

\vskip .1in
\begin{proof}[Proof of Proposition \ref{xxx2}]
Let $h\in \mathbb{R}^2$ and consider the evolution of
\begin{equation}\label{gdef}
g(x,t;h)= (\delta_h\theta(x,t))^2\, \Phi(h)
\end{equation}
where $\Phi(h) =e^{-\Psi(h)}$ with $\Psi$ satisfying
\begin{equation}\label{Psip}
\Psi(h)=\Psi(|h|)\,\,\mbox{ is nondecreasing}, \quad\Psi(0)=0, \quad \Psi(h)\to +\infty\quad\mbox{as}\quad |h|\to \infty.
\end{equation}
An explicit form of $\Psi$ will be specified later. First of all,
$\delta_h\theta$ satisfies
\begin{equation}\label{deltat}
(\pp_t + u\cdot\na + \delta_h u \cdot\na_h + \Lambda^\beta) \delta_h \th =0,
\end{equation}
where $\delta_h u(x,t)=u(x+h,t) - u(x,t)$. Multiplying (\ref{deltat}) by $\delta_h\theta(x,t)\, \Phi(h)$ and applying Lemma \ref{difflow},
we obtain
\begin{eqnarray*}
&& (\pp_t + u\cdot\na + \delta_h u \cdot\na_h + \Lambda^\beta) g
+ 2D_h(\delta_h \th) \,\Phi(h) + C\,\Phi(h)\, \frac{|\delta_h \th(x,t)|^{2+\beta}}{\|\theta\|_{L^\infty}^\beta |h|^\beta}\\
&& \qquad \le (\delta_h\theta(x,t))^2\, \delta_h u \cdot\na_h \Phi(h).
\end{eqnarray*}
Noticing that $\na_h \Phi(h) =-\Phi(h) \na_h \Psi(h)$, we have
$$
(\delta_h\theta(x,t))^2\, \delta_h u \cdot\na_h \Phi(h) = - (\delta_h\theta(x,t))^2\,\Phi(h)\,\delta_h u \cdot\na_h \Psi(h).
$$
Therefore, if we write $\mathcal{L} \equiv \pp_t + u\cdot\na + \delta_h u \cdot\na_h
+ \Lambda^\beta$, we obtain
\begin{equation} \label{Leq}
\mathcal{L} g + 2D_h(\delta_h \th) \,\Phi(h) + C\, \frac{g^{1+\frac\beta2}}{\Phi^{\frac{\beta}{2}}(h)\|\theta\|_{L^\infty}^\beta |h|^\beta}
\le g\, |\delta_h u|\, \Psi'(|h|).
\end{equation}
Thanks to the assumptions on $\widetilde{u}$,
$$
|\delta_h \widetilde{u}| \le \widetilde{M} \equiv \min\{M |h|, 2\|u\|_{L^\infty(0,T; L^\infty)}\} <\infty.
$$
Therefore, $u = \widetilde{u} + v$ satisfies
\begin{equation} \label{dhuu}
|\delta_h u| \le \widetilde{M} + |\delta_h v|.
\end{equation}
According to Lemma \ref{vrep},
$$
v(x,t) = C\,\int_{\mathbb{R}^2} \frac{(x-y)^\perp}{|x-y|^{1+\beta}}\, \pp_1 \th(y)\,dy.
$$
Since $\beta \in (0,1)$ and $\th$ is bounded and decays sufficiently fast at $\infty$,
we obtain after integration by parts,
\begin{eqnarray*}
&& v_1(x,t) = (1+\beta)C\, \int_{\mathbb{R}^2} \frac{(x_1-y_1)(x_2-y_2)}{|x-y|^{3+\beta}}\,
\th(y)\,dy,\\
&& v_2(x,y) =C\, \int_{\mathbb{R}^2}\left[\frac{1}{|x-y|^{1+\beta}} + \frac{(1+\beta)(x_1-y_1)^2}{|x-y|^{3+\beta}}\right]\,
\th(y)\,dy.
\end{eqnarray*}
Therefore,
\begin{eqnarray}
|\delta_h v| &\le&  C\,  \int_{\mathbb{R}^2} \frac{1}{|x-y|^{1+\beta}} |\delta_h \th(y,t)| \,dy \nonumber\\
 &\le& C\,  \int_{|x-y|\le 1} \frac{1}{|x-y|^{1+\beta}} |\delta_h \th(y,t)| \,dy
 + C\, \int_{|x-y|> 1} \frac{1}{|x-y|^{1+\beta}} |\delta_h \th(y,t)| \,dy \nonumber\\
  &\le& C\,\|\th_0\|_{L^\infty\cap L^2}.\label{dhuu1}
\end{eqnarray}
Inserting the estimates of (\ref{dhuu}) and (\ref{dhuu1}) in (\ref{Leq}) and using $\Phi(h)\le 1$, we obtain
\begin{equation}\label{Leq0}
\mathcal{L} g + C_1\, \frac{g^{1+\frac\beta2}}{\|\theta_0\|_{L^\infty}^\beta |h|^\beta}
\le C_2\,g\,(\widetilde{M} +\|\theta_0\|_{L^2\cap L^\infty}) \, \Psi'(|h|),
\end{equation}
where $C_1$ and $C_2$ are constants depending on $\beta$ only. Now we choose
$$
\Psi(|h|) = \frac{C_1\, (\delta_0)^\beta}{C_2(1-\beta)\|\theta_0\|_{L^\infty}^\beta
(\widetilde{M} +\|\theta_0\|_{L^2\cap L^\infty})} |h|^{1-\beta}.
$$
Certainly this choice satisfies (\ref{Psip}). Then (\ref{Leq0}) becomes
\begin{equation}\label{Leq2}
\mathcal{L} g  \le \frac{C_1\,g}{|h|^\beta \|\th_0\|_{L^\infty}^\beta}
(\delta_0^\beta -g^{\frac\beta2}).
\end{equation}
We can then conclude from (\ref{Leq2}) that
$$
g(x,t;h) \le \delta^2_0 \quad\mbox{for $|h|\le L$}.
$$
This can be proven by following a similar argument as in the proof of
Proposition \ref{xxx1}. More precisely, we multiply (\ref{Leq2}) by
$\widetilde{\Gamma}'(g)$ with $\widetilde{\Gamma}$ being a non-decreasing
smooth convex function on $[0,\infty)$ satisfying $\widetilde{\Gamma}(\rho)=0$ for
$0\le \rho\le \delta_0^2$ and positive for $\rho\ge\delta_0^2$. Thanks
to (\ref{lowbd}),
$$
\mathcal{L} \left(\widetilde{\Gamma}(g)\right) \le
\frac{C_1\,g}{|h|^\beta \|\th_0\|_{L^\infty}^\beta}
(\delta_0^\beta -g^{\frac\beta2})\,\widetilde{\Gamma}'(g) \le 0.
$$
By first estimating the $L^q$-norm of $\widetilde{\Gamma}(g)$ and then sending
$q\to\infty$, we obtain that the $L^\infty$-norm of $\widetilde{\Gamma}(g)$
is bounded by its initial $L^\infty$-norm. Since initially $g\le \delta_0^2$ and
thus the initial $L^\infty$-norm of $\widetilde{\Gamma}(g)$ is zero, we have
$L^\infty$-norm of $\widetilde{\Gamma}(g)$ is zero for all time. Therefore
$g\le \delta_0^2$ and
$$
|\delta_h \theta(x,t)| \le \delta_0 \, (\Phi(h))^{-\frac12}
\le \delta_0 \,e^{\frac12 \Psi(L)}\, \le 4 \delta_0.
$$
This completes the proof of Proposition \ref{xxx2}.
\end{proof}


\vskip .3in
\section{Besov spaces and a commutator estimate}
\label{BesovComm}

This section provides the definitions of some of the functional spaces and related facts to be used in the subsequent sections .  More details can be found in several books and many papers (see, e.g., \cite{BCD,BL,MWZ,RS,Tri}). In addition, we
prove a commutator estimate to be used extensively in the sections that follow.

\vskip .1in
To introduce the Besov spaces, we start with a few notation. $\mathcal{S}$ denotes
the usual Schwarz class and ${\mathcal S}'$ its dual, the space of
tempered distributions. ${\mathcal S}_0$ denotes a subspace of ${\mathcal
S}$ defined by
$$
{\mathcal S}_0 = \left\{ \phi\in {\mathcal S}: \,\, \int_{\mathbb{R}^d}
\phi(x)\, x^\gamma \,dx =0, \,|\gamma| =0,1,2,\cdots \right\}
$$
and ${\mathcal S}_0'$ denotes its dual. ${\mathcal S}_0'$ can be identified
as
$$
{\mathcal S}_0' = {\mathcal S}' / {\mathcal S}_0^\perp = {\mathcal S}' /{\mathcal P}
$$
where ${\mathcal P}$ denotes the space of multinomials. For each $j\in \mathbb{Z}$, we write
\begin{equation}\label{aj}
A_j =\left\{ \xi \in \mathbb{R}^d: \,\, 2^{j-1} \le |\xi| <
2^{j+1}\right\}.
\end{equation}
The Littlewood-Paley decomposition asserts the existence of a
sequence of functions $\{\Phi_j\}_{j\in \mathbb{Z}}\in {\mathcal S}$ such
that
$$
\mbox{supp} \widehat{\Phi}_j \subset A_j, \qquad
\widehat{\Phi}_j(\xi) = \widehat{\Phi}_0(2^{-j} \xi)
\quad\mbox{or}\quad \Phi_j (x) =2^{jd} \Phi_0(2^j x),
$$
and
$$
\sum_{j=-\infty}^\infty \widehat{\Phi}_j(\xi) = \left\{
\begin{array}{ll}
1&,\quad \mbox{if}\,\,\xi\in \mathbb{R}^d\setminus \{0\},\\
0&,\quad \mbox{if}\,\,\xi=0.
\end{array}
\right.
$$
Therefore, for a general function $\psi\in {\mathcal S}$, we have
$$
\sum_{j=-\infty}^\infty \widehat{\Phi}_j(\xi)
\widehat{\psi}(\xi)=\widehat{\psi}(\xi) \quad\mbox{for $\xi\in \mathbb{R}^d\setminus \{0\}$}.
$$
In addition, if $\psi\in {\mathcal S}_0$, then
$$
\sum_{j=-\infty}^\infty \widehat{\Phi}_j(\xi)
\widehat{\psi}(\xi)=\widehat{\psi}(\xi) \quad\mbox{for any $\xi\in
{\mathbb R}^d $}.
$$
That is, for $\psi\in {\mathcal S}_0$,
$$
\sum_{j=-\infty}^\infty \Phi_j \ast \psi = \psi
$$
and hence
$$
\sum_{j=-\infty}^\infty \Phi_j \ast f = f, \qquad f\in {\mathcal S}_0'
$$
in the sense of weak-$\ast$ topology of ${\mathcal S}_0'$. For
notational convenience, we define
\begin{equation}\label{del1}
\mathring{\Delta}_j f = \Phi_j \ast f, \qquad j \in {\mathbb Z}.
\end{equation}
\begin{define}
For $s\in {\mathbb R}$ and $1\le p,q\le \infty$, the homogeneous Besov
space $\mathring{B}^s_{p,q}$ consists of $f\in {\mathcal S}_0' $
satisfying
$$
\|f\|_{\mathring{B}^s_{p,q}} \equiv \|2^{js} \|\mathring{\Delta}_j
f\|_{L^p}\|_{l^q} <\infty.
$$
\end{define}

\vskip .1in
We now choose $\Psi\in {\mathcal S}$ such that
$$
\widehat{\Psi} (\xi) = 1 - \sum_{j=0}^\infty \widehat{\Phi}_j (\xi),
\quad \xi \in \mathbb{R}^d.
$$
Then, for any $\psi\in {\mathcal S}$,
$$
\Psi \ast \psi + \sum_{j=0}^\infty \Phi_j \ast \psi =\psi
$$
and hence
\begin{equation}\label{sf}
\Psi \ast f + \sum_{j=0}^\infty \Phi_j \ast f =f
\end{equation}
in ${\mathcal S}'$ for any $f\in {\mathcal S}'$. To define the inhomogeneous Besov space, we set
\begin{equation} \label{del2}
\Delta_j f = \left\{
\begin{array}{ll}
0,&\quad \mbox{if}\,\,j\le -2, \\
\Psi\ast f,&\quad \mbox{if}\,\,j=-1, \\
\Phi_j \ast f, &\quad \mbox{if} \,\,j=0,1,2,\cdots.
\end{array}
\right.
\end{equation}

\begin{define}
The inhomogeneous Besov space $B^s_{p,q}$ with $1\le p,q \le \infty$
and $s\in {\mathbb R}$ consists of functions $f\in {\mathcal S}'$
satisfying
$$
\|f\|_{B^s_{p,q}} \equiv \|2^{js} \|\Delta_j f\|_{L^p} \|_{l^q}
<\infty.
$$
\end{define}

\vskip .1in
The Besov spaces $\mathring{B}^s_{p,q}$ and $B^s_{p,q}$ with  $s\in (0,1)$ and $1\le p,q\le \infty$ can be equivalently defined by the norms
$$
\|f\|_{\mathring{B}^s_{p,q}}  = \left(\int_{\mathbb{R}^d} \frac{(\|f(x+t)-f(x)\|_{L^p})^q}{|t|^{d+sq}} dt\right)^{1/q},
$$
$$
\|f\|_{B^s_{p,q}}  = \|f\|_{L^p} + \left(\int_{\mathbb{R}^d} \frac{(\|f(x+t)-f(x)\|_{L^p})^q}{|t|^{d+sq}} dt\right)^{1/q}.
$$
When $q=\infty$, the expressions are interpreted in the normal way.


\vskip .1in
Besides the Fourier localization operators $\Delta_j$, the partial sum $S_j$ is also a useful notation. For an integer $j$,
$$
S_j \equiv \sum_{k=-1}^{j-1} \Delta_k,
$$
where $\Delta_k$ is given by (\ref{del2}). For any $f\in \mathcal{S}'$, the Fourier transform of $S_j f$ is supported on the ball of radius $2^j$.

\vskip .1in
Bernstein's inequality is a useful tool on Fourier localized functions and these inequalities trade integrability for derivatives. The following proposition provides Bernstein type inequalities for fractional derivatives.
\begin{prop}\label{bern}
Let $\alpha\ge0$. Let $1\le p\le q\le \infty$.
\begin{enumerate}
\item[1)] If $f$ satisfies
$$
\mbox{supp}\, \widehat{f} \subset \{\xi\in \mathbb{R}^d: \,\, |\xi|
\le K 2^j \},
$$
for some integer $j$ and a constant $K>0$, then
$$
\|(-\Delta)^\alpha f\|_{L^q(\mathbb{R}^d)} \le C_1\, 2^{2\alpha j +
j d(\frac{1}{p}-\frac{1}{q})} \|f\|_{L^p(\mathbb{R}^d)}.
$$
\item[2)] If $f$ satisfies
\begin{equation*}\label{spp}
\mbox{supp}\, \widehat{f} \subset \{\xi\in \mathbb{R}^d: \,\, K_12^j
\le |\xi| \le K_2 2^j \}
\end{equation*}
for some integer $j$ and constants $0<K_1\le K_2$, then
$$
C_1\, 2^{2\alpha j} \|f\|_{L^q(\mathbb{R}^d)} \le \|(-\Delta)^\alpha
f\|_{L^q(\mathbb{R}^d)} \le C_2\, 2^{2\alpha j +
j d(\frac{1}{p}-\frac{1}{q})} \|f\|_{L^p(\mathbb{R}^d)},
$$
where $C_1$ and $C_2$ are constants depending on $\alpha,p$ and $q$
only.
\end{enumerate}
\end{prop}

\vskip .1in
The rest of this section is devoted to the proof of a commutator estimate. We need the
following lemma.
\begin{lemma} \label{ccces}
Let $\delta \in (0,1)$. Let $q, \, q_1, \,q_2, r_1, r_2 \in [1, \infty]$ satisfying
$$
\frac1{q_1} + \frac1 {q_2} = \frac1q , \qquad  \frac1{r_1} + \frac1{r_2} =1.
$$
If $f\in \mathring{B}^\delta_{q_1, r_1}(\mathbb{R}^d)$, $|x|^{\delta+\frac{d}{r_1}} \phi \in L^{r_2}(\mathbb{R}^d)$,  and $g\in L^{q_2}(\mathbb{R}^d)$, then
\begin{equation}\label{conin}
\|\phi\ast ( fg) - f \phi\ast g\|_{L^q} \le C\, \||x|^{\delta+\frac{d}{r_1}} \phi\|_{L^{r_2}} \, \|f\|_{\mathring{B}^\delta_{q_1, r_1}} \,\|g\|_{L^{q_2}},
\end{equation}
where $C$ is a constant independent of $f, g$ and $h$.  In the borderline case when $\delta =1$, (\ref{conin}) still holds if $\|f\|_{\mathring{B}^\delta_{q_1, r_1}}$ is replaced by $\|\nabla f\|_{L^{q_1}}$ and $r_2=1$.
\end{lemma}

The case when $\delta=1$, $q_1=q$ and $q_2=\infty$ was obtained in \cite[p.2153]{HKR1}. The fractional case with $\delta\in (0,1)$, $q_1=q$ and $q_2=\infty$ was obtained by Chae and Wu in \cite{ChaeWu}. Here we need to make use of the more general case when $q_1>q$.

\vskip .1in
\begin{proof}[Proof of Lemma \ref{ccces}]
By Minkowski's inequality followed by H\"{o}lder's inequality, for any $q\in [1,\infty]$,
\begin{eqnarray*}
\|\phi\ast (fg) -f (\phi\ast g)\|_{L^q}
&=&\left[\int \left|\int \phi(z) \, (f(x) -f(x-z)) g(x-z) \,dz\right|^q \,dx\right]^{1/q} \\
&\le& \int |\phi(z)| \, \|f(\cdot)-f(\cdot-z))\|_{L^{q_1}} \,\|g\|_{L^{q_2}} dz\\
&\le& \|g\|_{L^{q_2}} \left\| \frac{\|f(\cdot)-f(\cdot-z))\|_{L^{q_1}}}{|z|^{\delta+\frac{d}{r_1}}} \right\|_{L^{r_1}}\, \left\||z|^{\delta+\frac{d}{r_1}} |\phi(z)| \right\|_{L^{r_2}}
\end{eqnarray*}
\eqref{conin} then follows from the definition of $\mathring{B}^\delta_{q_1,r_1}$.
\end{proof}

\vskip .1in
\begin{prop} \label{cmmu}
Let $\alpha\in (0,1)$. Let $s\in (0,1)$ and $\delta\in (0,1)$ satisfy $s+1-\alpha- \delta<0$.
Let $q\in [2, \infty)$, $r\in [1, \infty]$, $q_1, q_2\in [2, \infty]$ satisfy $\frac1q =\frac1{q_1} + \frac1{q_2}$. Then, for a constant $C$,
\begin{equation} \label{cmmb}
\|[\mathcal{R}_\alpha, f] g\|_{B^{s}_{q,r}} \le C\, \|f\|_{{B}^\delta_{q_1, \infty}}\, \|g\|_{B^{s+1-\alpha-\delta}_{q_2,r}}.
\end{equation}
When $\delta=1$, (\ref{cmmb}) is still valid if $\|f\|_{B^\delta_{q_1, \infty}}$ is replaced by $\|\nabla f\|_{L^{p_1}}$.
\end{prop}

\begin{proof}
Let $k\ge -1$ be an integer. By the notion of paraproducts,
we write
$$
\Delta_k [\mathcal{R}_\alpha, f] g  = J_1 + J_2 + J_3,
$$
where
\begin{eqnarray*}
J_1 &=& \sum_{|j-k|\le 2} \Delta_k \left(\mathcal{R}_\alpha (S_{j-1} f \, \Delta_j g) - S_{j-1} f \, \mathcal{R}_\alpha \Delta_j g \right), \\
J_2 &=& \sum_{|j-k|\le 2} \Delta_k \left(\mathcal{R}_\alpha (\Delta_j f\, S_{j-1} g) -\Delta_j  f \, \mathcal{R}_\alpha S_{j-1} g\right), \\
J_3 &=& \sum_{j\ge k-1} \Delta_k \left(\mathcal{R}_\alpha (\Delta_j f\, \widetilde{\Delta}_j  g) -\Delta_j  f \, \mathcal{R}_\alpha \widetilde{\Delta}_j g \right)
\end{eqnarray*}
with $\widetilde{\Delta}_j = \Delta_{j-1} + \Delta_j + \Delta_{j+1}$. We first note that, if the Fourier transform of $F$ is supported in the annulus around radius $2^j$, then $\mathcal{R}_\alpha F $ can be represented as a convolution,
\begin{equation}\label{rare}
\mathcal{R}_\alpha F = h_j \ast F, \qquad h_j(x) = 2^{(d+1-\alpha) j}\, h_0(2^j x)
\end{equation}
for a function $h_0$ in the Schwartz class $\mathcal{S}$. This can be obtained by simply examining the Fourier transform of $\mathcal{R}_\alpha F$. By the definition of $B^s_{q,r}$,
\begin{eqnarray}
\left\|[\mathcal{R}_\alpha, f] g \right\|_{B^s_{q,r}} &=& \left\|2^{sk} \|\Delta_k [\mathcal{R}_\alpha, f] g\|_{L^q} \right\|_{l^r} \nonumber \\
&\le& \left\|2^{sk} \|J_1\|_{L^q} \right\|_{l^r} + \left\|2^{sk} \|J_2\|_{L^q} \right\|_{l^r} + \left\|2^{sk} \|J_3\|_{L^q} \right\|_{l^r}. \label{root}
\end{eqnarray}
Applying Lemma \ref{ccces}, we find
\begin{eqnarray*}
\|J_1\|_{L^q} &\le& C\, 2^{(1-\alpha) k}\,\||x|^\delta  2^{dk}  h_0(2^k x)\|_{L^1} \, \|S_{j-1} f\|_{\mathring{B}^\delta_{q_1, \infty}}\, \|\Delta_j g\|_{L^{q_2}} \nonumber\\
&\le& C\, 2^{(1-\alpha-\delta) k}\,\|f\|_{\mathring{B}^\delta_{q_1, \infty}}\, \|\Delta_j g\|_{L^{q_2}}.
\end{eqnarray*}
Thus,
\begin{eqnarray}
\left\|2^{sk} \|J_1\|_{L^q} \right\|_{l^r} \le C\, \|f\|_{\mathring{B}^\delta_{q_1, \infty}}\, \|g\|_{B^{s+1-\alpha-\delta}_{q_2,r}}.\label{j1es}
\end{eqnarray}
By Bernstein's inequality, we have
\begin{eqnarray*}
\|J_2\|_{L^q} &\le& C\, 2^{(1-\alpha-\delta)j}
\|f\|_{{B}^\delta_{q_1, \infty}}\, \|S_{j-1} g\|_{L^{q_2}} \nonumber\\
&\le& C\, 2^{(1-\alpha-\delta)j}\,\|f\|_{{B}^\delta_{q_1, \infty}}\,
\sum_{m\le k-2} \|\Delta_m g\|_{L^{q_2}}\\
&\le& C\, \|f\|_{{B}^\delta_{q_1, \infty}}\,
\sum_{m\le k-2} 2^{(1-\alpha-\delta) (k-m)}\, 2^{(1-\alpha-\delta) m}\|\Delta_m g\|_{L^{q_2}}.
\end{eqnarray*}
Since  $s+1-\alpha-\delta<0$, we obtain, by applying Young's inequality for series,
\begin{eqnarray}
\left\|2^{sk} \|J_2\|_{L^q} \right\|_{l^r} \le C \, \|f\|_{{B}^\delta_{q_1, \infty}}\, \|g\|_{B^{s+1-\alpha-\delta}_{q_2,r}}. \label{j2es}
\end{eqnarray}
Similarly, we have
\begin{eqnarray*}
\|J_3\|_{L^q} &\le& C\, \sum_{j\ge k-1} 2^{(1-\alpha-\delta) j}\,\|f\|_{{B}^\delta_{q_1, \infty}}\,\|\Delta_j g\|_{L^{q_2}}.
\end{eqnarray*}
Therefore, for $s>0$, by Young's inequality for series,
\begin{eqnarray}
\left\|2^{sk} \|J_3\|_{L^q} \right\|_{l^r} &\le&  C\, \left\|2^{sk} \sum_{j\ge k-1} 2^{(1-\alpha-\delta) j}\,\|f\|_{B^\delta_{q_1, \infty}}\,\|\Delta_j g\|_{L^{q_2}}\right\|_{l^r} \nonumber \\
&\le& C\, \|f\|_{{B}^\delta_{q_1, \infty}}\, \|g\|_{B^{s+1-\alpha-\delta}_{q_2,r}}.\label{j3es}
\end{eqnarray}
Combining (\ref{root}), (\ref{j1es}), (\ref{j2es}) and (\ref{j3es}), we obtain the desired bound in (\ref{cmmb}). This completes the proof of Proposition \ref{cmmu}.
\end{proof}

\vskip .1in
We have used a simple inequality, namely (\ref{simin}) in the proof of
Theorem \ref{glohs} and we now prove it. Its proof is deferred to this
section since it involves the Littlewood-Paley decomposition and
Bernstein's inequality described above. By Bernstein's inequality,
\begin{eqnarray*}
\|\na \nabla^\perp \Lambda^{-3+\beta} \pp_{1} \theta^\epsilon\|_{L^\infty}
&\le& \|\Delta_{-1} \na \nabla^\perp \Lambda^{-3+\beta} \pp_{1} \theta^\epsilon\|_{L^\infty}+ \sum_{j=0}^\infty \|\Delta_{j} \na \nabla^\perp \Lambda^{-3+\beta} \pp_{1} \theta^\epsilon\|_{L^\infty}\\
&\le& C\, \|\Delta_{-1} \theta^\epsilon\|_{L^2}
+ \sum_{j=0}^\infty  2^{\beta j} \|\Delta_j\theta^\epsilon\|_{L^\infty} \\
&\le& C\, \|\theta^\epsilon\|_{L^2}
+ \sum_{j=0}^\infty  2^{(\beta-1) j} \|\nabla \Delta_j\theta^\epsilon\|_{L^\infty},
\end{eqnarray*}
which reduces to (\ref{simin}) by recalling that $\beta\in(0,1)$.

\vskip .4in
\section{Global $L^2$-bound for $G$}
\label{GL2}

This section establishes a global {\it a priori} bound for $\|G\|_{L^2}$.
Recall that
\begin{equation} \label{Gdd}
G= \omega - \mathcal{R}_\alpha \theta \quad\mbox{with} \quad \mathcal{R}_\alpha = \Lambda^{-\alpha}\partial_{1}
\end{equation}
and $G$ satisfies
\begin{equation} \label{Geq0}
\partial_t G + u\cdot\nabla G + \Lambda^\alpha G = [\mathcal{R}_\alpha, u\cdot\nabla] \theta + \Lambda^{1-2\alpha} \pp_1 \th.
\end{equation}
This global bound is valid for any $\frac45<\alpha<1$.

\begin{thm} \label{l2b}
Assume $(u_0, \th_0)$ satisfies the assumptions stated in Theorem \ref{main}. Let $(u,
\th)$ be the corresponding solution of (\ref{BQE}). If $\frac45 < \alpha < 1$, then $G$ defined in
(\ref{Gdd}) has a global $L^2$ bound, namely for any $T>0$ and $t\le T$,
$$
\|G(t)\|_{L^2}^2 + \int_0^t \|\Lambda^{\frac{\alpha}{2}} G(\tau)\|_{L^2}^2 \,d\tau \le  C (T, u_0, \theta_0),
$$
where $C (T, u_0, \theta_0)$ is a constant depending on $T$ and $\|u_0\|_{L^2}$ and $\|\theta_0\|_{L^2\cap L^\infty}$ only.
\end{thm}

\vskip .1in
To prove Theorem \ref{l2b}, the following elementary global {\it a priori} bounds will be used.
\begin{lemma}
Assume that $u_0$ and $\theta_0$ satisfy the conditions stated in Theorem \ref{main}.
Then the corresponding solution $(u, \th)$ of (\ref{BQE}) obey the global bounds,
for any $t>0$,
$$
\|\th(t)\|_{L^r} \le \|\th_0\|_{L^r} \qquad \mbox{for any $2\le r\le \infty$},
$$ $$
\|u(t)\|^2_{L^2} + \int_0^t \|\Lambda^{\frac{\alpha}{2}} u(\tau)\|_{L^2}^2 \,d\tau
\le \left(\|u_0\|^2_{L^2} + t \|\th_0\|^2_{L^2}\right)^2.
$$
\end{lemma}

\vskip .1in
\begin{proof}[Proof of Theorem \ref{l2b}]  Taking the inner product of (\ref{Geq0}) with $G$, we obtain, after integration by parts,
\begin{equation}\label{g2root}
\frac12 \frac{d}{dt} \|G\|_{L^2}^2  + \|\Lambda^{\frac{\alpha}{2}} G\|_{L^2}^2  = J_1 + J_2,
\end{equation}
where
\begin{eqnarray*}
&& J_1 =\int G\, [\mathcal{R}_\alpha, u\cdot\nabla] \theta \,dx = \int G\, \nabla \cdot [\mathcal{R}_\alpha, u] \theta \,dx,\\
&& J_2 = \int G\,\Lambda^{1-2\alpha} \pp_1 \th\,dx.
\end{eqnarray*}
Let $q_1, q_2\in [2,\infty]$ satisfying $\frac1{q_1}+ \frac1{q_2}=\frac12$. By H\"{o}lder's inequality and Proposition \ref{cmmu},
\begin{eqnarray*}
|J_1| &\le& \|\Lambda^{\frac{\alpha}{2}} G\|_{L^2} \, \|\Lambda^{1-\frac{\alpha}{2}} \Lambda^{-1} \nabla\cdot [\mathcal{R}_\alpha, u] \theta\|_{L^2} \\
&\le& C\, \|\Lambda^{\frac{\alpha}{2}} G\|_{L^2}\|u\|_{{B}^\alpha_{q_1, \infty}}\, \|\theta\|_{B^{2-\frac52\alpha}_{q_2, 2}}.
\end{eqnarray*}
Setting $q_1 = \frac{2}{\alpha}$ and noticing that
$$
u =\na^\perp (-\Delta)^{-1} \om = \na^\perp (-\Delta)^{-1} (G + \Lambda^{-\alpha} \pp_1\th),
$$
we obtain, by the Hardy-Littlewood-Sobolev inequality,
\begin{eqnarray}
\|u\|_{\mathring{B}^\alpha_{q_1, \infty}} &\le& C\, \|\Lambda^\alpha u\|_{L^{q_1}}
\le  C\, \|\Lambda^{\alpha-1} \omega\|_{L^{q_1}} \nonumber \\ &\le&C\, \|\Lambda^{\alpha-1} G\|_{L^{q_1}} + C\,\|\theta\|_{L^{q_1}} \nonumber\\ &\le& C\, \|G\|_{L^2} + C\, \|\theta_0\|_{L^{q_1}}. \label{lau}
\end{eqnarray}
So, we get
$$
\|u\|_{{B}^\alpha_{q_1, \infty}}\le \|u\|_{\mathring{B}^\alpha_{q_1, \infty}}+C\|u\|_{L^2}\le C\|G\|_{L^2}
+ C\,\|\theta_0\|_{L^{q_1}}+C\|u\|_{L^2}.
$$
Thanks to the embedding $L^{q_2} \hookrightarrow B^{2-\frac52\alpha}_{q_2, 2}$ when $\alpha>\frac45$, we have
\begin{eqnarray}
 \|\theta\|_{B^{2-\frac52\alpha}_{q_2, 2}} \le C \, \|\theta_0\|_{L^{q_2}}. \label{lau2}
\end{eqnarray}
Therefore, by Young's inequality
$$
|J_1| \le \frac14 \, \|\Lambda^{\frac{\alpha}{2}} G\|_{L^2}^2  \, + C\, \|\theta_0\|_{L^{q_2}}^2 \|G\|^2_{L^2} + C\,\|\theta_0\|_{L^{q_2}}^2\left(\|\theta_0\|_{L^{q_1}}^2 + \|u_0\|^2_{L^2} + t^2 \|\th_0\|_{L^2}^2\right),
$$
where $C$'s are pure constants. Since $\frac45<\alpha<1$,
\begin{eqnarray*}
|J_2| &\le& \|\th_0\|_{L^2}\, \|\Lambda^{2-2\alpha} G\|_{L^2} \\
&\le& \|\th_0\|_{L^2}\,
\|\Lambda^{\frac{\alpha}{2}} G\|^{4(\frac1{\alpha}-1)}_{L^2} \, \|G\|^{5-\frac{4}{\alpha}}_{L^2}\\
&\le& \frac14 \, \|\Lambda^{\frac{\alpha}{2}} G\|_{L^2}^2  \, + C\,\|\th_0\|^{\frac{2\alpha}{5\alpha-4}}_{L^2} \|G\|^2_{L^2}.
\end{eqnarray*}
Inserting the bounds for $J_1$ and $J_2$ in (\ref{g2root}) and applying Gronwall's inequality yield the desired bound. This completes the proof of Theorem \ref{l2b}.
\end{proof}

\vskip .4in
\section{Global $L^q$-bound for $G$ with $2<q<q_0$}
\label{GLq}

This section proves a global bound for $\|G\|_{L^q}$ with $2<q<q_0$, where $q_0$ is specified in (\ref{q0}). This global bound is valid for any $\frac45<\alpha<1$.

\begin{thm}\label{lqb}
Let $\frac45 < \alpha < 1$ and $\alpha+\beta=1$. Consider (\ref{BQE}) with $u_0 \in H^2$ and $\theta_0\in B^2_{2,1}$. Let $(u, \theta)$ be the corresponding smooth solution. Assume that $2< q < q_0$ with $q_0$ given by
\begin{equation} \label{q0}
q_0 \equiv \frac{8-4\alpha}{8-7\alpha}.
\end{equation}
Then, for any $T>0$ and $t\le T$,
$$
\|G(t)\|_{L^q} \le  C (T, u_0, \theta_0),
$$
where $C (T, u_0, \theta_0)$ is a constant depending on $T$, $u_0$ and $\theta_0$ only.
\end{thm}

\begin{proof}
Taking the inner product of (\ref{Geq0}) with $G |G|^{q-2}$, we have
\begin{equation}\label{rot}
\frac1q \frac{d}{dt} \|G\|_{L^q}^q  + \int G |G|^{q-2} \Lambda^\alpha G \,dx = K_1 + K_2,
\end{equation}
where
\begin{eqnarray*}
&& K_1 = \int G |G|^{q-2} \,[\mathcal{R}_\alpha, u\cdot\nabla]\theta\,dx,\\
&& K_2 = \int G |G|^{q-2}\, \Lambda^{1-2\alpha} \pp_1 \th\,dx.
\end{eqnarray*}
By a pointwise inequality for fractional Laplacians (see \cite{CC})
and a Sobolev embedding inequality,
$$
\int G |G|^{q-2} \Lambda^\alpha G \,dx \ge  C\, \int |\Lambda^{\frac{\alpha}{2}} |G|^{\frac{q}{2}}|^2 \,dx  \ge C_0\, \|G\|_{L^{\frac{2q}{2-\alpha}}}^q,
$$
where $C, C_0>0$ are constants. To bound $K_1$, we have for any $s\in (0,1)$,
\begin{eqnarray*}
|K_1| &\le& \left|\int \Lambda^s(G |G|^{q-2})\, \Lambda^{1-s}(\Lambda^{-1} \nabla\cdot [\mathcal{R}_\alpha, u]\theta)\,dx \right|\\
&\le& \|\Lambda^s(G |G|^{q-2})\|_{L^2} \,\, \|\Lambda^{1-s}([\mathcal{R}_\alpha, u]\theta) \|_{L^2}.
\end{eqnarray*}
By Proposition \ref{cmmu}, (\ref{lau}) and (\ref{lau2}), we obtain
\begin{eqnarray*}
\|\Lambda^{1-s} [\mathcal{R}_\alpha, u]\theta\|_{L^2} &=& \|[\mathcal{R}_\alpha, u]\theta\|_{\mathring{H}^{1-s}} \\
&\le& C\, \|u\|_{B^{\alpha}_{q_1, \infty}}\, \|\theta\|_{B^{2-s-2\alpha}_{q_2,2}} +C \, \|u\|_{L^2}\,\|\theta_0\|_{L^2}\\
&\le& C\,(\|G\|_{L^2} + \|\theta_0\|_{L^{q_1}}+\|u\|_{L^2})\|\th_0\|_{L^{q_2}} + C \, \|u\|_{L^2}\,\|\theta_0\|_{L^2} \equiv B(t),
\end{eqnarray*}
where $2-2\alpha<s$, $q_1 =\frac{2}{\alpha}$, $\frac1{q_1} +\frac1{q_2}=\frac12$ and
$B(t)$ is a smooth function of $t$ depending on $\|u_0\|_{L^2}$ and
$\|\th_0\|_{L^2\cap L^{q_1}\cap L^{q_2}}$ only.  Choosing $s$ satisfying
$$
2+s-\alpha-\frac{2(2-\alpha)}{q} =\frac{\alpha}{2}
$$
and applying Lemma \ref{goody} below, we obtain
\begin{eqnarray*}
\|\Lambda^s (G |G|^{q-2})\|_{L^2} \le C\, \|G\|^{q-2}_{L^{\frac{2q}{2-\alpha}}} \,\|G\|_{\mathring{H}^{2+s-\alpha-\frac{2(2-\alpha)}{q}}} = C\, \|G\|^{q-2}_{L^{\frac{2q}{2-\alpha}}}\,\|G\|_{\mathring{H}^{\frac{\alpha}{2}}}.
\end{eqnarray*}
Combining the estimates, we obtain, for any $t\le T$,
\begin{eqnarray*}
|K_1| \le B(t)\, \|G\|^{q-2}_{L^{\frac{2q}{2-\alpha}}}\,\|G\|_{\mathring{H}^{\frac{\alpha}{2}}}
\le \frac{C_0}{2} \|G\|^{q}_{L^{\frac{2q}{2-\alpha}}} + C \, \|G\|^{\frac{q}{2}}_{\mathring{H}^{\frac{\alpha}{2}}},
\end{eqnarray*}
where $C$ is a constant depending on $T$ and the norms of the initial data. Since
$2-2\alpha<s$,
\begin{eqnarray*}
|K_2| &\le& \|(1+\Lambda)^s (G |G|^{q-2})\|_{L^2} \|(1+\Lambda)^{2-2\alpha-s} \th\|_{L^2}\\
 &\le& C(\|\Lambda^s (G |G|^{q-2})\|_{L^2}+\|G\|_{L^2}^{q-1})\, \|\th_0\|_{L^2} \\
&\le& \frac{C_0}{2} \|G\|^{q}_{L^{\frac{2q}{2-\alpha}}} + C \, \|G\|^{\frac{q}{2}}_{\mathring{H}^{\frac{\alpha}{2}}}+C\|G\|_{L^2}^{q-1}\|\th_0\|_{L^2}.
\end{eqnarray*}
Inserting the estimates for $K_1$ and $K_2$ in (\ref{rot}) and using the fact in
Theorem \ref{l2b} that,
for $\frac{q}{2} \le 2$ and any $T>0$,
$$
\int_0^T \|G\|^{\frac{q}{2}}_{\mathring{H}^{\frac{\alpha}{2}}} \,dt <\infty,
$$
we then obtain the desired the global bound. The conditions on the indices
$$
q\ge 2, \qquad 2-s-2\alpha<0, \quad 2+s-\alpha-\frac{2(2-\alpha)}{q} =\frac{\alpha}{2}
$$
are fulfilled if
$$
\alpha >\frac45, \qquad 2\le q< q_0 \equiv \frac{8-4\alpha}{8-7\alpha}.
$$
This completes the proof of Theorem \ref{lqb}.
\end{proof}

\vskip .1in
We have used the following lemma in the proof of Theorem \ref{lqb}. This lemma
generalizes a previous inequality of \cite[p.2170]{HKR1}.
\begin{lemma} \label{goody}
Let $s\in (0,1)$, $\alpha\in (0,1)$ and $q\in[2,\infty)$. Then,
$$
\|\Lambda^s (G |G|^{q-2})\|_{L^2} \le C \|G\|^{q-2}_{L^{\frac{2q}{2-\alpha}}} \, \|G\|_{\mathring{B}^s_{\frac{2q}{2(2-\alpha)-q(1-\alpha)},2}}.
$$
Especially,
\begin{equation}\label{ert}
\|\Lambda^s (G |G|^{q-2})\|_{L^2} \le C \|G\|^{q-2}_{L^{\frac{2q}{2-\alpha}}} \, \|G\|_{\mathring{H}^{2+s-\alpha-\frac{2(2-\alpha)}{q}}}.
\end{equation}
\end{lemma}

\begin{proof}
Identifying $\mathring{H}^s$ with the Besov space $\mathring{B}^s_{2,2}$, we have
$$
\|\Lambda^s (G |G|^{q-2})\|^2_{L^2} = \int \frac{\|G |G|^{q-2}(\cdot -z) -G |G|^{q-2}(\cdot)\|_{L^2}^2}{|z|^{2+2s}} \,dz.
$$
Using the simple inequality, for $q\ge 2$,
$$
|a |a|^{q-2} -b |b|^{q-2}| \le C\,|a-b| (|a|^{q-2} + |b|^{q-2})
$$
and H\"{o}lder's inequality, we have
$$
\|G |G|^{q-2}(\cdot -z) -G |G|^{q-2}(\cdot)\|_{L^2}^2 \le C \, \|G\|^{2(q-2)}_{L^{\frac{2q}{2-\alpha}}}\, \|G(\cdot -z) -G(\cdot)\|_{L^{\frac{2q}{2(2-\alpha)-q(1-\alpha)}}}^2.
$$
Therefore,
$$
\|\Lambda^s (G |G|^{q-2})\|^2_{L^2} \le C\, \|G\|^{2(q-2)}_{L^{\frac{2q}{2-\alpha}}}\,
\|G\|^2_{\mathring{B}^s_{\frac{2q}{2(2-\alpha)-q(1-\alpha)},2}}
$$
(\ref{ert}) holds due to the embedding $\mathring{B}^{2+s-\alpha-\frac{2(2-\alpha)}{q}}_{2,2} \hookrightarrow \mathring{B}^s_{\frac{2q}{2(2-\alpha)-q(1-\alpha)},2}$. This completes the proof of
Lemma \ref{goody}.
\end{proof}

\vskip .4in
\section{Proof of Theorem \ref{main} }

This section proves Theorem \ref{main}. To do so, we need more regularity for $G$.
By fully exploiting the dissipation, we are able to show in
Proposition \ref{smoth} below that
$G$ in $B^s_{q,\infty}$ with any $s\le 3\alpha-2$ and $2\le q<q_0$ is
bounded for all time.

\vskip .1in
\begin{prop} \label{smoth}
Consider the IVP (\ref{BQE}) with $\alpha_0<\alpha<1$. Assume the initial data $(u_0, \th_0)$ satisfies the conditions stated in Theorem \ref{main} and let $(u, \th)$ be
the corresponding solution. If the indices $s$ and $q$ satisfy
\begin{equation}\label{sqdef}
0< s \le 3\alpha -2, \quad \frac2{2\alpha-1}<q < q_0 \equiv \frac{8-4\alpha}{8-7\alpha},
\end{equation}
then $G$ obeys the global {\it a priori} bound, for any $T>0$,
$$
\|G\|_{L^\infty(0,T; B^s_{q,\infty})} \le C,
$$
where $C$ is a constant depending on $T$ and the initial norms only.
\end{prop}

\vskip .1in
This proposition will be proven at the end of this section. With this proposition at our disposal, we are ready to prove Theorem \ref{main}.

\begin{proof}[Proof of Theorem \ref{main}]
The key part of the proof is to establish the global {\it a priori} bounds of the solution
in the functional setting (\ref{solutionclass}). According to Proposition \ref{smoth},
for any $T>0$, and $s$ and $q$ satisfying (\ref{sqdef}), we have
\begin{equation}\label{Greg}
G \equiv \om -\Lambda^{-\alpha} \pp_1 \theta \in L^\infty(0,T; B^s_{q,\infty}),
\end{equation}
Since $u=\na^\perp \Delta^{-1} \om$, we have
$$
u = \na^\perp \Delta^{-1}\,G + \na^\perp \Delta^{-1}\,\Lambda^{-\alpha} \pp_1 \theta
\equiv \widetilde{u} + v.
$$
Since $\alpha_0<\alpha<1$, we can choose $q$ satisfying (\ref{sqdef}) such that
$$
\frac{2}{q} - (3\alpha-2) <0.
$$
In fact, thanks to $\alpha_0<\alpha$, we have
$$
\frac{2}{3\alpha-2} < \frac{8-4\alpha}{8-7\alpha}
$$
and any $q$ satisfies
$$
\frac{2}{3\alpha-2} < q < q_0 \equiv \frac{8-4\alpha}{8-7\alpha}
$$
would work.  By the embedding $B^0_{\infty,1} \hookrightarrow L^\infty$, Bernstein's inequality and the
boundedness of Riesz transforms on $L^q$,
\begin{eqnarray*}
\|\na \widetilde{u}\|_{L^\infty} &\le& \sum_{j=-1}^\infty \|\na \na^\perp \Delta^{-1} \Delta_j G\|_{L^\infty} \\
&\le& C\, \sum_{j=-1}^\infty 2^{j \frac{2}{q}} \|\Delta_j G\|_{L^q}\\
&\le& \|G\|_{B^s_{q,\infty}}\, \sum_{j=-1}^\infty 2^{j (\frac{2}{q} -s)}.
\end{eqnarray*}
Taking $s \le 3\alpha-2$ but close to $3\alpha-2$, we have $\frac{2}{q} -s<0$ and thus
$$
M \equiv  \|\na \widetilde{u}\|_{L^\infty(0, T;L^\infty)} <\infty.
$$
Applying Theorem \ref{glohs} to the equation for $\theta$,
$$
\begin{cases}
\pp_t \th + u \cdot \na \th + \Lambda^{1-\alpha} \th =0, \\
u = \widetilde{u} + v, \quad v\equiv -\na^\perp \Delta^{-1}\,\Lambda^{-\alpha} \pp_1 \theta,
\end{cases}
$$
we obtain
\begin{equation} \label{nathbbb1}
\th\in C([0,T]; H^1), \quad \|\nabla \th\|_{L^\infty(0, T; L^\infty)} \le C(M, \|\th_0\|_{L^\infty}, \|\na \th_0\|_{L^\infty},T) <\infty.
\end{equation}
According to the vorticity equation
\begin{equation} \label{vvor}
\pp_t \om + u\cdot \na\om + \Lambda^\alpha \om = \pp_1\th,
\end{equation}
we have, thanks to $\om_0\in B^{\sigma-1}_{2,1}(\mathbb{R}^2)$ with $\sigma\ge \frac52$ and $B^1_{2,1}(\mathbb{R}^2)\hookrightarrow L^\infty(\mathbb{R}^2)$,
\begin{equation} \label{nathbbb2}
\|\om(t)\|_{L^\infty} \le \|\om_0\|_{L^\infty}
+ \int_0^T \|\pp_1\th(\tau)\|_{L^\infty}\,d\tau.
\end{equation}
(\ref{nathbbb1}) and (\ref{nathbbb2}) will lead
us to the global {\it a priori} bounds in (\ref{solutionclass}).
To this goal, we employ the
following estimate, for any integer $j\ge -1$ and any $t\in(0,T]$,
\begin{eqnarray}
\frac{2^{\alpha j}}{j+2} \int_0^t \|\Delta_j \om\|_{L^\infty} \,d\tau
&\le& C\,\left(\|\om_0\|_{L^\infty} + \int_0^t \|\pp_1\th(\tau)\|_{L^\infty}\,d\tau\right) \nonumber\\
&& \times \left(1+ t +\|u\|_{L^2} + \|\om\|_{L^1_tL^\infty}\right)\equiv h(t), \label{bddd}
\end{eqnarray}
where $C$ is a constant independent of $j$. (\ref{bddd}) can be established by applying
Besov space type estimates on (\ref{vvor}) and more details can be found
in \cite[p.134]{MWZ}. As a consequence, we have, for any $0<\alpha'<\alpha$,
\begin{eqnarray}
\int_0^t \sum_{j=-1}^\infty 2^{\alpha' j}\,\|\Delta_j \om\|_{L^\infty}\,d\tau
&=& \sum_{j=-1}^\infty 2^{\alpha' j}\,\int_0^t \|\Delta_j \om\|_{L^\infty}\,d\tau\nonumber\\
&\le& h(t) \, \sum_{j=-1}^\infty 2^{(\alpha'-\alpha) j}/(j+1) \le C\, h(t). \label{opop}
\end{eqnarray}
Therefore, by Bernstein's inequality and $\|\Delta_j \nabla u\|_{L^\infty} \le C\,\|\Delta_j \om\|_{L^\infty}$ for $j\ge 0$,
\begin{eqnarray}
\|\na u\|_{L^\infty} \le \|\Delta_{-1} \na u\|_{L^\infty} + \sum_{j=0}^\infty \|\Delta_j \nabla u\|_{L^\infty} \le C \|u\|_{L^2}
+ C\,\sum_{j=0}^\infty \|\Delta_j \om\|_{L^\infty}. \label{opop1}
\end{eqnarray}
Combining (\ref{opop}) and (\ref{opop1}) yields, for any $0<t\le T$,
\begin{eqnarray}
\int_0^t \|\nabla u(\tau)\|_{L^\infty}\,d\tau < \infty. \label{nuuu}
\end{eqnarray}
Through standard Besov space type estimates, we can show that
$$
Y(t) \equiv \|u(t)\|_{B^\sigma_{2,1}} + \|\th(t)\|_{B^2_{2,1}}
$$
obeys the differential inequality
$$
\frac{d}{dt} Y(t) + C\,(\|u\|_{B^{\sigma+\alpha}_{2,1}} + \|\th\|_{B^{2+\beta}_{2,1}})
\le C\,(1+ \|\na u(t)\|_{L^\infty} + \|\na\th(t)\|_{L^\infty}) Y(t).
$$
Gronwall's inequality together with (\ref{nathbbb1}) and (\ref{nuuu}) yields
the desired property in (\ref{solutionclass}). The continuity of $u$ and
$\th$ in time, namely $u\in C([0,T];B^\sigma_{2,1})$ and
$\th\in C([0,T];B^2_{2,1})$, follows from a
standard approach (see, e.g., \cite[p.111]{MB} or \cite[p.138]{BCD}).
We thus have obtained all the global {\it a priori} bounds in (\ref{solutionclass}).
The global existence part can then be obtained by first establishing
the local existence through a standard process such as the successive
approximation and then extending to all time via the global {\it a priori}
bounds. The uniqueness of solutions
in the functional setting (\ref{solutionclass}) is clear and we thus omit
the details. This
completes the proof of Theorem \ref{main}.
\end{proof}

\vskip .1in
We finally prove Proposition \ref{smoth}.

\begin{proof}[Proof of Proposition \ref{smoth}]
Recall that $G$ satisfies
\begin{equation} \label{Geq00}
\partial_t G + u\cdot\nabla G + \Lambda^\alpha G = [\mathcal{R}_\alpha, u\cdot\nabla] \theta + \Lambda^{1-2\alpha} \pp_1 \th.
\end{equation}
Let $j\ge -1$. Applying $\Delta_j$ to (\ref{Geq00}) and taking the inner product with $\Delta_j G |\Delta_j G|^{q-2}$, we obtain
\begin{equation}\label{qroot}
\frac1q \frac{d}{dt} \|\Delta_j G\|_{L^q}^q + \int \Delta_j G |\Delta_j G|^{q-2} \Lambda^\alpha \Delta_j G \,dx = K^{(j)}_1 + K^{(j)}_2+ K^{(j)}_3,
\end{equation}
where
\begin{eqnarray}
K^{(j)}_1 &=& \int \Delta_j G |\Delta_j G|^{q-2}\, \Delta_j [\mathcal{R}_\alpha, u\cdot\nabla] \theta \,dx,\nonumber\\
K^{(j)}_2 &=& - \int \Delta_j G |\Delta_j G|^{q-2} \Delta_j (u\cdot\nabla G)\,dx, \label{k2j}\\
K^{(j)}_3 &=& \int \Delta_j G |\Delta_j G|^{q-2} \Delta_j
(\Lambda^{1-2\alpha} \pp_1 \th)\,dx.\nonumber
\end{eqnarray}
For $j\ge 0$, the Fourier transform of $\Delta_j G$ is supported away from the origin and the dissipative part admits a lower bound,
\begin{equation}\label{qlow}
\int \Delta_j G |\Delta_j G|^{q-2} \Lambda^\alpha \Delta_j G \,dx  \ge C\, 2^{\alpha j} \|\Delta_j G\|_{L^q}^q
\end{equation}
for a constant $C$ that depends on $q$ and $\alpha$ only (see, e.g. \cite{CMZ,Wu1}). For $j=-1$, the dissipative part is still nonnegative and can be neglected.  By H\"{o}lder's inequality,
$$
|K^{(j)}_1| \le \|\Delta_j G\|_{L^q}^{q-1} \, \|\Delta_j [\mathcal{R}_\alpha, u\cdot\nabla] \theta\|_{L^q}.
$$
Furthermore, $\|\Delta_j [\mathcal{R}_\alpha, u\cdot\nabla] \theta\|_{L^q}$ can be estimated in a similar fashion as in the proof of Proposition \ref{cmmu} and is bounded by
$$
\|\Delta_j [\mathcal{R}_\alpha, u\cdot\nabla] \theta\|_{L^q} \le C \, 2^{(2-2\alpha) j}\, \|u\|_{\mathring{B}^\alpha_{q,\infty}}\, \|\theta\|_{L^\infty} + \,C\,\|u\|_{L^2}\,\|\theta\|_{L^2}.
$$
In addition, as in (\ref{lau}),
\begin{eqnarray}
\|u\|_{\mathring{B}^\alpha_{q,\infty}}\, &\le&C\, \|\Lambda^\alpha u\|_{L^q} \le C\,\|\Lambda^{\alpha-1} G\|_{L^q} + C\, \|\theta\|_{L^q} \nonumber \\
&\le&  C\, \|G\|_{L^{\widetilde{q}}} \, +\,  C\, \|\theta_0\|_{L^q}, \label{lamu}
\end{eqnarray}
where $\frac1{\widetilde{q}}=\frac1{q} -\frac{\alpha}{2} +\frac12$. For $\alpha\in (\alpha_0,1)$, we have $2<\widetilde{q} <q$ and thus $\|G\|_{L^{\widetilde{q}}} <C$. Therefore,
\begin{equation}\label{k1bb}
|K^{(j)}_1| \le  B(t)\, 2^{(2-2\alpha) j}\, \|\Delta_j G\|_{L^q}^{q-1},
\end{equation}
where $B_1(t)$ is a smooth function of $t$ that depends on the initial norms only.
To estimate $K^{(j)}_2$, we apply the notion of paraproducts to write
\begin{eqnarray}
\Delta_j (u\cdot\nabla G) = J_1 + J_2 + J_3 + J_4 + J_5,  \label{dec}
\end{eqnarray}
where
\begin{eqnarray*}
J_1 &=& \sum_{|j-k|\le 2} \left(\Delta_j(S_{k-1} u \cdot \nabla \Delta_k G) -S_{k-1} u \cdot \nabla \Delta_j\Delta_k G\right),\\
J_2  &=& \sum_{|j-k|\le 2} (S_{k-1} u - S_ju) \cdot\nabla \Delta_j\Delta_k G, \\
J_3  &=& S_j u\cdot\nabla \Delta_j G, \\
J_4 &=&  \sum_{|j-k|\le 2} \Delta_j (\Delta_k u\cdot\nabla S_{k-1} G),\\
J_5 &=& \sum_{k\ge j-1} \Delta_j(\Delta_k u \cdot\nabla \widetilde{\Delta}_k G)
\end{eqnarray*}
with $\widetilde{\Delta}_k=\Delta_{k-1} + \Delta_k + \Delta_{k+1}$. Inserting (\ref{dec}) in (\ref{k2j}) naturally splits the integral in $K_2^{(j)}$
into five parts
$K_2^{(j)} = K_{21}^{(j)} + K_{22}^{(j)} + K_{23}^{(j)} + K_{24}^{(j)} + K_{25}^{(j)}$. By H\"{o}lder's inequality,
$$
|K_{21}^{(j)}| \le \|\Delta_j G\|_{L^q}^{q-1}\, \sum_{|j-k|\le 2} \left\|(\Delta_j(S_{k-1} u \cdot \nabla \Delta_k G) -S_{k-1} u \cdot \nabla \Delta_j\Delta_k G)\right\|_{L^q}.
$$
Since the summation above is for $k$ satisfying $|j-k|\le 2$ and can be replaced by a constant multiple of the representative term with $k=j$, we obtain, by a standard commutator estimate and Bernstein's inequality,
\begin{eqnarray*}
|K_{21}^{(j)}| &\le& C\, \|\Delta_j G\|_q^{q-1}\, 2^{(1-\alpha)j}\,\|\Lambda^\alpha S_{j-1} u\|_{L^q}\, \|\Delta_j G\|_{L^\infty} \\
&\le& C\, \|\Delta_j G\|_q^{q-1}\, 2^{(1-\alpha+\frac{2}{q})j}\,\|\Lambda^\alpha u\|_{L^q}\, \|\Delta_j G\|_{L^q}\\
&\le& C\, 2^{(1-\alpha+\frac{2}{q})j}\,\|\Lambda^\alpha u\|_{L^q}\, \|\Delta_j G\|_{L^q}^q.
\end{eqnarray*}
The second term $K_{22}^{(j)}$ can be bounded by
\begin{eqnarray*}
|K_{22}^{(j)}| &\le& C\, \|\Delta_j G\|_{L^q}^{q-1}\,\|\Delta_j u\|_{L^q} \, 2^j \|\Delta_j G\|_{L^\infty}. \\
&\le& C\,2^{(1-\alpha+\frac{2}{q})j}\, \|\Lambda^\alpha\Delta_j u\|_{L^q} \,\|\Delta_j G\|_{L^q}^q,
\end{eqnarray*}
where we have used the lower bound part of the Bernstein inequality
$$
2^{\alpha j}\,\|\Delta_j u\|_{L^q} \le C\,\|\Lambda^\alpha\Delta_j u\|_{L^q}.
$$
This inequality is valid for  $j \ge 0$. In the case when $j=-1$, this inequality is not needed and it suffices to apply the upper bound part of the Bernstein inequality
$$
\|\Delta_{-1} u\|_{L^q}  \le C\, \|\Delta_{-1} u\|_{L^2}.
$$
Due to the divergence-free condition $\nabla\cdot u =0$, we have $K_{23}^{(j)}=0$. By H\"{o}lder's inequality and Bernstein's inequality
\begin{eqnarray*}
|K_{24}^{(j)}| &\le&  C \, \|\Delta_j G\|_{L^q}^{q-1}\,
\|\Delta_j u\cdot\nabla S_{j-1} G\|_{L^q} \\
&\le&  C \, 2^{-\alpha j}\, \|\Delta_j G\|_{L^q}^{q-1}\,\|\Lambda^\alpha \Delta_j u\|_{L^q}\, \sum_{m\le j-1} 2^{(1+\frac2q)m}\, \|\Delta_m G\|_{L^q} \\
&\le&  C\,2^{(1-\alpha+\frac{2}{q})j}\,\|\Delta_j G\|_{L^q}^{q-1}\,\|\Lambda^\alpha \Delta_j u\|_{L^q}\, \sum_{m\le j-1} 2^{(1+\frac2q)(m-j)}\, \|\Delta_m G\|_{L^q}.
\end{eqnarray*}
\begin{eqnarray*}
|K_{25}^{(j)}| &\le&  C \, \|\Delta_j G\|_{L^q}^{q-1}\, \sum_{k\ge j-1}  2^j \|\Delta_j (\Delta_k u \widetilde{\Delta}_k G)\|_{L^q} \\
 &\le&  C \, 2^{(1-\alpha+\frac{2}{q})j}\,\|\Delta_j G\|_{L^q}^{q-1}\, \sum_{k\ge j-1}  2^{(j-k)(\alpha-\frac2q)} \,\|\Lambda^\alpha \Delta_k u\|_{L^q} \, \|\widetilde{\Delta}_k G\|_{L^q}.
\end{eqnarray*}
Collecting the estimates for $K_2^{(j)}$ and bounding $\|\Lambda^\alpha u\|_{L^q}$ as in \eqref{lamu}, we have
\begin{eqnarray}
|K_2^{(j)}| &\le&  C\,2^{(1-\alpha+\frac{2}{q})j}\,\|\Delta_j G\|_{L^q}^{q-1}\, \Big[\|\Delta_j G\|_{L^q} + \sum_{m\le j-1} 2^{(1+\frac2q)(m-j)}\, \|\Delta_m G\|_{L^q}
\nonumber\\
&&+ \sum_{k\ge j-1}  2^{(j-k)(\alpha-\frac2q)}  \, \|\Delta_k G\|_{L^q}\Big]. \label{k2bb}
\end{eqnarray}
By H\"{o}lder's inequality and Bernstein's inequality,
\begin{eqnarray}
|K_3^{(j)}| \le 2^{2(1-\alpha)j} \, \|\Delta_j G\|_{L^q}^{q-1}\, \|\Delta_j \th\|_{L^q}
\le 2^{2(1-\alpha)j} \, \|\Delta_j G\|_{L^q}^{q-1}\, \|\th_0\|_{L^q}.\label{k3bb}
\end{eqnarray}
Combining (\ref{qroot}), (\ref{qlow}), (\ref{k1bb}), (\ref{k2bb})
and (\ref{k3bb}), we obtain
\begin{eqnarray*}
&& \frac{d}{dt} \|\Delta_j G\|_{L^q} + C\, 2^{\alpha j} \|\Delta_j G\|_{L^q}  \le C\, B(t)\,2^{(2-2\alpha) j}\,+ \, C\,2^{(1-\alpha+\frac{2}{q})j}\, L(t),
\end{eqnarray*}
where $B(t)$ is a smooth function of $t$ depending on the initial norms only and
\begin{eqnarray*}
L=\|\Delta_j G\|_{L^q}+ \sum_{m\le j-1} 2^{(1+\frac2q)(m-j)}\, \|\Delta_m G\|_{L^q} +  \sum_{k\ge j-1}  2^{(j-k)(\alpha-\frac2q)}  \, \|\Delta_k G\|_{L^q}.
\end{eqnarray*}
Integrating in time leads to
\begin{eqnarray*}
\|\Delta_j G(t)\|_{L^q} &\le& e^{-C 2^{\alpha j} t}\, \|\Delta_j G(0)\|_{L^q} + C\, \widetilde{B}(t) \,2^{(2-3\alpha) j} \\
&& + \,C\,2^{(1-\alpha+\frac{2}{q})j}\, \int_0^t e^{-C 2^{\alpha j} (t-\tau)} L(\tau)\,d\tau,
\end{eqnarray*}
where $\widetilde{B}(t)$ is a smooth function of $t$ depending on the initial norms only. Multiplying each side by $2^{sj}$ with
$0< s\le 3\alpha-2$ and taking sup with respect to $j$, we find that
\begin{eqnarray}
\|G(t)\|_{B^{s}_{q,\infty}} \le \|G(0)\|_{B^{s}_{q,\infty}} + C\,\widetilde{B}(t) + L_1 + L_2 + L_3,
\label{ges}
\end{eqnarray}
where
\begin{eqnarray*}
&& L_1 = \sup_{j\ge -1} \,C\,2^{(1-\alpha+\frac{2}{q})j}\, \int_0^t e^{-C 2^{\alpha j} (t-\tau)} 2^{sj} \|\Delta_j G(\tau)\|_{L^q} \,d\tau, \\
&& L_2 = \sup_{j\ge -1} \,C\,2^{(1-\alpha+\frac{2}{q})j}\, \int_0^t e^{-C 2^{\alpha j} (t-\tau)} 2^{sj} \sum_{m\le j-1} 2^{(1+\frac2q)(m-j)}\, \|\Delta_m G\|_{L^q}\,d\tau, \\
&& L_3 = \sup_{j\ge -1} \,C\,2^{(1-\alpha+\frac{2}{q})j}\, \int_0^t e^{-C 2^{\alpha j} (t-\tau)} 2^{sj} \sum_{k\ge j-1}  2^{(j-k)(\alpha-\frac2q)}  \, \|\Delta_k G\|_{L^q}\,d\tau.
\end{eqnarray*}
Since $1-2\alpha +\frac{2}{q} <0$,  we choose an integer $j_0$ such that, for $j \ge j_0$,
$$
C\, 2^{(1-2\alpha+\frac{2}{q})j} \le \frac18.
$$
Then
\begin{eqnarray*}
L_1 &\le& \sup_{-1 \le j\le j_0} \,C\,2^{(1-\alpha+\frac{2}{q})j}\, \int_0^t e^{-C 2^{\alpha j} (t-\tau)} 2^{sj} \|\Delta_j G(\tau)\|_{L^q} \,d\tau \\
&& + \, \sup_{j_0+1\le j<\infty} \,C\,2^{(1-\alpha+\frac{2}{q})j}\, \int_0^t e^{-C 2^{\alpha j} (t-\tau)} 2^{sj} \|\Delta_j G(\tau)\|_{L^q} \,d\tau\\
&\le& C\,\|G\|_{L^\infty(0,T; L^q)}\, 2^{(1-2\alpha+\frac{2}{q}+s)j_0} + \frac18 \|G(t)\|_{L^\infty(0,T; B^{s}_{q,\infty})}
\end{eqnarray*}
$L_2$ and $L_3$ obey the same bound. Inserting these bounds in (\ref{ges}), we find
$$
\|G\|_{L^\infty(0,T; B^{s}_{q,\infty})}
\le C\,(1+\|G(0)\|_{B^{s}_{q,\infty}}) + C\,\|G\|_{L^\infty(0,T; L^q)},
$$
where we have written $C$ for $\max_{0\le t\le T}\widetilde{B}(t)$. Since $u_0 \in B^\sigma_{2,1}$ with $\sigma\ge \frac52$ and $\th_0\in B^2_{2,1}$,
$$
\om_0 \in B^{\sigma-1}_{2,1}(\mathbb{R}^2)\hookrightarrow B^s_{q,\infty}(\mathbb{R}^2), \quad G(0) =\om_0 - \Lambda^{-\alpha} \pp_1 \th_0 \in B^{s}_{q,\infty}.
$$
This finishes the
proof of Proposition \ref{smoth}.
\end{proof}

\vskip .4in
\section*{Acknowledgements}
Jiu was supported by NSF of China under grants No.11171229 and No.11231006. Miao was supported by NSF of China under grants No.11171033 and No.11231006.   Wu was partially supported by NSF grant DMS1209153. Zhang was supported by NSF of China under grants No.10990013 and No.11071007, Program for New Century University Talents and Fok Ying Tung Education Foundation. In addition, Jiu and Wu were partially supported by a special grant from NSF of China under No.11228102.

\end{document}